\documentclass[a4paper,11pt]{amsart}
\usepackage{amsmath, amsthm, amsfonts, amssymb}
\usepackage{mathtools}
\usepackage{enumerate}
\usepackage[bookmarks=false]{hyperref} 
\usepackage{color}

\title{Properly proximal groups and their von Neumann algebras}
\author{R\'emi Boutonnet, Adrian Ioana and Jesse Peterson}
\address{CNRS -- Institut de Math\'{e}matiques de Bordeaux, Universit\'{e} de Bordeaux, 351 cours de la Lib\'{e}ration, 33 405 Talence Cedex, France}
\email{remi.boutonnet@math.u-bordeaux.fr}
\address{Department of Mathematics, University of California San Diego, 9500 Gilman Drive, La Jolla, CA 92093, USA}
\email{aioana@ucsd.edu}
\address{Department of Mathematics, Vanderbilt University, 1326 Stevenson Center, Nashville, TN 37240, USA}
\email{jesse.d.peterson@vanderbilt.edu}
\newtheorem{thm}{Theorem}[section]
\newtheorem{prop}[thm]{Proposition}
\newtheorem{cor}[thm]{Corollary}
\newtheorem{lem}[thm]{Lemma}
\theoremstyle{definition}
\newtheorem{defn}[thm]{Definition}
\newtheorem{defn/lem}[thm]{Definition/Lemma}
\newtheorem{rem}[thm]{Remark}
\newtheorem{examp}[thm]{Example}

\newtheorem{problem}{Problem}[]

\newtheorem{note}[thm]{Notation}
\newtheorem{que}[problem]{Question}

\newcommand{\C}{{\mathbb C}}
\newcommand{\F}{{\mathbb F}}
\newcommand{\G}{{\mathbb G}}

\newcommand{\K}{{\mathbb K}}

\newcommand{\N}{{\mathbb N}}
\newcommand{\bP}{{\mathbb P}}
\newcommand{\Q}{{\mathbb Q}}
\newcommand{\R}{{\mathbb R}}
\newcommand{\bS}{{\mathbb S}}

\newcommand{\V}{{\mathbb V}}
\newcommand{\Z}{{\mathbb Z}}

\newcommand{\cC}{{\mathcal C}}
\newcommand{\cB}{{\mathcal B}}
\newcommand{\cF}{{\mathcal F}}
\newcommand{\cG}{{\mathcal G}}

\newcommand{\cJ}{{\mathcal J}}

\newcommand{\cL}{{\mathcal L}}

\newcommand{\cN}{{\mathcal N}}

\newcommand{\cS}{{\mathcal S}}
\newcommand{\cU}{{\mathcal U}}

\newcommand{\fL}{\mathfrak{L}}

\newcommand{\bQ}{{\overline{\Q}}}

\newcommand{\txi}{{\tilde{\xi}}}

\newcommand{\tpsi}{{\tilde{\psi}}}

\newcommand{\opi}{\overline{\pi}}
\newcommand{\oH}{\overline{H}}
\newcommand{\oxi}{\overline{\xi}}

\newcommand{\Ad}{\operatorname{Ad}}

\newcommand{\Out}{\operatorname{Out}}

\newcommand{\Prob}{\operatorname{Prob}}

\newcommand{\GL}{\operatorname{GL}}
\newcommand{\SL}{\operatorname{SL}}
\newcommand{\PSL}{\operatorname{PSL}}
\newcommand{\SO}{\operatorname{SO}}

\newcommand{\conv}{\operatorname{conv}}
\newcommand{\diag}{\operatorname{diag}}

\newcommand{\ot}{\otimes}
\newcommand{\ovt}{\, \overline{\otimes}\,}
\newcommand{\mini}{{\operatorname{min}}}

\newcommand{\eps}{\varepsilon}
\newcommand{\dpr}{^{\prime\prime}}

\newcommand{\op}{{\rm op}}
\newcommand{\actson}{{\, \curvearrowright \,}}

\begin{document}
\begin{abstract}
We introduce a wide class of countable groups, called properly proximal, which contains all non-amenable bi-exact groups, all non-elementary convergence groups, and all lattices in non-compact semi-simple Lie groups, but excludes all inner amenable groups.  
We show that crossed product II$_1$ factors arising from free ergodic
probability measure preserving actions of groups in this class have at most one weakly compact Cartan subalgebra, up to unitary conjugacy. As an application, we obtain the first $W^*$-strong rigidity results for compact actions of $SL_d(\mathbb Z)$ for $d \geq 3$.

\end{abstract}

\maketitle
%\tableofcontents

%%%%%%%%%%
\section{Introduction}
Countable groups and their measure preserving actions naturally give rise to von Neumann algebras, via two constructions of Murray and von Neumann \cite{MvN36,MvN43}. 
This work is motivated by the following general problem: prove structural results for the von Neumann algebras associated with the  arithmetic groups $\SL_d(\Z)$, $d \geq 3$, and their probability measure preserving ({\it p.m.p.}) actions. At present, relatively little is known in this direction. Thus, nearly all available results regarding the group von Neumann algebras $L(\SL_d(\Z))$, $d \geq 3$, are either direct consequences of property (T) \cite{Co80}, or concern inclusions $L\Lambda\subset L(\SL_d(\Z))$ for some subgroups $\Lambda<\SL_d(\Z)$, rather than $L(\SL_d(\Z))$ itself \cite{BC14}. Moreover, while several remarkable rigidity results for crossed product von Neumann algebras associated to actions of $\SL_d(\Z)$ have been obtained in \cite{Po03,Po04,Io10,Bo12},  these are restricted to specific classes of actions.

In contrast, the structure of von Neumann algebras associated with $\Gamma := \SL_2(\Z)$ and its actions is much better understood. 
Indeed, from the perspective of deformation/rigidity theory there has been a lot of work in this direction, starting with two seminal results obtained in the early 2000s. First, Popa used his deformation/rigidity theory to show that the crossed product von Neumann algebra $L^{\infty}(X)\rtimes\Gamma$ associated to any free ergodic p.m.p.\ action $\Gamma\curvearrowright (X,\mu)$ has at most one Cartan subalgebra with the relative property (T) \cite{Po01}. Second, Ozawa employed C$^*$-algebraic techniques to prove that  $L\Gamma$ is {\it solid}: the relative commutant, $A'\cap L\Gamma$, of any diffuse von Neumann subalgebra $A\subset L\Gamma$ is amenable \cite{Oz04}. 
These results have since been considerably strengthened, also in the context of Popa's deformation/rigidity theory, following two breakthroughs of Ozawa and Popa \cite{OP10a} and  Popa and Vaes \cite{PV12a}:

\begin{enumerate}
\item  $L\Gamma$ is \emph{strongly solid}: the normalizer of any diffuse amenable von Neumann subalgebra $A \subset L\Gamma$ generates an amenable von Neumann algebra. Moreover,  $L^\infty(X,\mu) \rtimes \Gamma$ admits $L^\infty(X,\mu)$ as its unique Cartan subalgebra, up to unitary conjugacy, for any free ergodic compact p.m.p. action $\Gamma\curvearrowright (X,\mu)$ (see \cite{OP10a}).
\item $\Gamma$ is  \emph{$\cC$-rigid}: $L^\infty(X,\mu) \rtimes \Gamma$ admits $L^\infty(X,\mu)$ as its unique Cartan subalgebra, up to unitary conjugacy, for any  free ergodic p.m.p. action $\Gamma \actson (X,\mu)$ (see \cite{PV12a}). In particular, $L^{\infty}(X)\rtimes\Gamma$ entirely remembers the orbit equivalence relation  of the action $\Gamma\curvearrowright (X,\mu)$ \cite{FM77b}.
\end{enumerate}
Recall that a Cartan subalgebra of a tracial von Neumann algebra $M$ is a maximal abelian subalgebra $A\subset M$ whose normalizer generates $M$.

In fact, in the last 15 years, a plethora of impressive structural results have been obtained for von Neumann algebras arising from large classes of countable groups $\Gamma$ and their measure preserving actions (see \cite{Oz06A, Po07, Va10, Io17}). 
However, in most of these results, some negative curvature condition on $\Gamma$ is needed, in the form of a geometric assumption (e.g., $\Gamma$ is a hyperbolic group or a lattice in a rank one simple Lie group \cite{Oz04, PV12b}), or a cohomological assumption (e.g., $\Gamma$ has positive first $\ell^2$-Betti number, \cite{Pe09, PS12, chifanpeterson, CS13, Va13}), or an algebraic assumption (e.g., $\Gamma$ is an  amalgamated free product group, \cite{IPP08, CH10, PV09, Io15}). In sharp contrast, lattices in higher rank simple Lie groups, such as $\SL_d(\Z)$ for $d\geq 3$,  
do not satisfy  any reasonable notion of negative curvature. 

The results (1) and (2) were generalized in \cite{CS13} and \cite{PV12b} to any group $\Gamma$ which is both \emph{weakly amenable} \cite{CH89,Ha16} and  \emph{bi-exact} (equivalently, belongs to Ozawa's class $\cS$) \cite{Sk88,Oz04,Oz06}. 
The proofs of statements (1) and (2) for such groups $\Gamma$ split into two parts.
 First, one uses the weak amenability of $\Gamma$ to deduce that any amenable subalgebra of $L(\Gamma)$ or $L^{\infty}(X)\rtimes\Gamma$ satisfies a certain weak compactness property (\cite{OP10a}, see Definition \ref{WC}). This fact is then combined with the bi-exactness of $\Gamma$ to prove the desired conclusions. The weak amenability and bi-exactness properties are enjoyed by hyperbolic groups and lattices in simple Lie groups of rank one. However,  both of these properties fail dramatically for lattices in higher rank simple Lie groups.
 
One of the main goals of this paper is to generalize the bi-exactness methods to a broader class of groups. The class of groups admitting proper cocycles into nonamenable representations was already considered in \cite[Theorem A]{OP10b}, and products of such groups were considered in \cite[Section 4]{CS13}, however, the methods therein do not apply to general higher rank lattices such as $\SL_d(\Z)$ for $d \geq 3$. The following is our first main result:

\begin{thm}\label{toy thm}
For any $d \geq 2$, the von Neumann algebra of $\Gamma := \SL_d(\Z)$ does not admit a weakly compact Cartan subalgebra. Moreover, for any free ergodic p.m.p.\ action $\sigma: \Gamma \actson (X,\mu)$, the crossed product $L^\infty(X,\mu) \rtimes \Gamma$ admits a weakly compact Cartan subalgebra $A$ if and only if $\sigma$ is weakly compact and, in this case, $A$ is unitary conjugate with $L^\infty(X,\mu)$.
\end{thm}

Combining Theorem \ref{toy thm} with \cite{Io11} we obtain that the following corollary.

\begin{cor}\label{profinite} Let $\sigma:\SL_d(\Z)\curvearrowright (X,\mu)$ and $\sigma':\SL_{d'}(\Z)\curvearrowright (X',\mu')$ be free ergodic profinite p.m.p. actions, for some $d,d'\geq 3$. If $L^{\infty}(X)\rtimes\SL_d(\Z)$ is isomorphic to $L^{\infty}(X')\rtimes\SL_{d'}(\Z)$, then $d=d'$ and the actions  $\sigma$ and  $\sigma'$ are virtually conjugate.
%For $d \neq d'$, the von Neumann crossed-products associated with a free profinite action of $\SL_d(Z)$ and any action of $\SL_{d'}(Z)$ are never isomorphic.
\end{cor}

\begin{rem}
For $d\geq 3$ and a non-empty set of primes $\mathcal P$, consider the left translation action of $\SL_d(\Z)$ on the compact group $K_{d,\mathcal P}:=\prod_{p\in\mathcal P}\SL_d(\Z_p)$ endowed with its Haar measure, where $\Z_p$ denotes the ring of $p$-adic integers. Corollary \ref{profinite} implies that  $L^{\infty}(K_{d,\mathcal P})\rtimes\SL_d(\Z)$ and $L^{\infty}(K_{d',\mathcal P'})\rtimes\SL_{d'}(\Z)$ are isomorphic  if and only if $(d,\mathcal P)=(d',\mathcal P')$.
\end{rem}

The proof of Theorem \ref{toy thm} is based on topological dynamics. Typically we use the dynamics of the actions of $\SL_d(\Z)$ on the projective space $\bP^{d-1}$ and other Grassmanian varieties. Note that these actions are neither topologically amenable nor small at infinity, as required in the definition of bi-exactness.
Instead, we exploit the fact that these actions do not admit invariant probability measures in combination with their proximality properties. To this end, we develop a general method to construct a nice compactification  (or rather, a piece of a compactification, see Definition~\ref{defn:piece}) of $\Gamma$ out of a continuous action on a compact space. We then use the framework developed in \cite{BC14} to exploit this compactification. 

It turns out that the above strategy applies to a much larger class of groups, which we call \emph{Properly Proximal Groups}.
Roughly speaking, we say that a group $\Gamma$ is properly proximal if it admits finitely many ``non-trivial'' continuous actions $\Gamma\actson K_i$ on compact spaces such that any sequence in $\Gamma$ admits a subsequence which is ``proximal'' for at least one of the actions $\Gamma\curvearrowright K_i$ (see Definition \ref{patched cv} for the precise definition). %Here, we call an action $\Gamma\curvearrowright K$ non-trivial if it does not admit an invariant probability measure, and say that a sequence $(g_n)$ in $\Gamma$ proximal if it pushes the whole orbit $\Gamma\cdot\mu$ of some given probability measure $\mu \in \Prob(K)$ towards a single measure on $K$. See Definition \ref{patched cv} for the precise statement.

Generalizing Theorem \ref{toy thm}, we prove the following result.

\begin{thm}\label{wc cartan}
Any properly proximal group $\Gamma$ satisfies the conclusions of Theorem \ref{toy thm}.
\end{thm}

Under the additional assumption that $\Gamma$ is weakly amenable, we obtain the following strengthening of Theorem~\ref{wc cartan}.

\begin{thm}\label{C-rigid} 
Let $\Gamma$ be a properly proximal, weakly amenable group. Then  $L\Gamma$ has no Cartan subalgebra and $\Gamma$ is $\cC$-rigid.
\end{thm}

We devote a substantial part of the paper to study the class of properly proximal groups. In particular we prove the following results.

\begin{prop}
Groups in the following classes are properly proximal:
\begin{itemize}
\item Non-amenable bi-exact groups;
\item Non-elementary convergence groups;
\item Lattices in non-compact semi-simple Lie groups of arbitrary rank;
\item Groups admitting a proper cocycle into a non-amenable representation.
\end{itemize}
Moreover, the class of properly proximal groups is stable under commensurability up to finite kernels and under direct products. In contrast, properly proximal groups are never inner amenable, and therefore no infinite direct product of non-trivial groups is properly proximal.
\end{prop}

We conclude the introduction with several questions on properly proximal groups.

\begin{que}
\begin{enumerate}[(a)]
\item Are mapping class groups properly proximal? What about outer automorphism groups $\Out(\F_n)$ of the free groups?
\item More generally, is the class of properly proximal groups invariant under measure equivalence? 
\item Is there a non inner amenable group which is not properly proximal? As we discuss in Section \ref{linear}, we suspect that  finitely generated linear groups are properly proximal if and only if they are not inner amenable.
\end{enumerate}
\end{que}

%\textcolor{red}{JP:}Further groups to consider: Acylindrically hyperbolic groups (mapping class groups, Out$(\mathbb F_n)$, Coxeter groups), Thompson's groups $T$, $V$, Monod's group of piecewise $SL_2(\mathbb R)$ transformations, groups with positive first $\ell^2$-Betti number, groups with non-zero harmonic cocycles into mixing representations, Kac-Moody groups, Baumslag-Solitar groups.  
%\textcolor{blue}{RB} Baumslag solitaire groups are inner amenable by a result of Stalder from 2006.

\subsection{Organization of the paper}
Apart from the introduction, this paper contains five other sections. Section 2 sets the notations and gives some preliminary facts. In Section 3, we develop the notion of a boundary piece and give the main constructions from dynamical systems and from cocycles. % (and more general arrays). 
In Section 4, we define and study properly proximal groups and prove the main results cited above.
Then, in Section 5 we show that boundary pieces may be used to define a notion of ``directional bi-exactness'', and generalize \cite[Section 15]{BO08} to this setting. Some concrete applications to the von Neumann algebras of $\SL_d(\Z)$ are given in Section 6.

\subsection{Acknowledgements}
The first named author is grateful to Jean-Fran\c{c}ois Quint for an inspiring discussion on compactifications of symmetric spaces and for many helpful discussions on algebraic groups over local fields. Parts of this project have been done during the conference ``Von Neumann algebras and measured group Theory'' held at the Institut Henri Poincar\'e, Paris, in July 2017, the workshop ``Classification of group von Neumann algebras'' held at the American Institute of Mathematics, San Jose, CA, in January 2018 and the workshop ``Approximation Properties in Operator Algebras and Ergodic Theory'' held at the Institute for Pure and Applied Mathematics, UCLA, CA, during the 2018 spring trimester ``Quantitative Linear algebra''. We thank these institutes for their hospitality and the organizers for the invitations to these events. We warmly thank Narutaka Ozawa for sharing Theorem \ref{carac pp} with us and for useful comments.

%%%%%%%%%%
\section{Notation}

\subsection*{Groups, actions and representations} 
The symbol $\Gamma$ will always refer to a discrete group. Given such a group, we denote by $\lambda$ and $\rho$ its left and right regular representations, respectively, both of them acting on $\ell^2\Gamma$. The canonical basis of $\ell^2\Gamma$ will be denoted by $\{\delta_g\}, g \in \Gamma$.
We will denote by $\Delta\Gamma$ the Stone-\v{C}ech compactification of $\Gamma$ and by $\partial\Gamma := \Delta \Gamma \setminus \Gamma$ its Stone-\v{C}ech boundary.

When considering a compact space $K$ we will denote by $\Prob(K)$ the set of regular Borel probability measures on $K$. We will often consider actions $\Gamma \actson K$. Implicitly we assume that such actions are continuous. Such actions naturally induce other actions $\Gamma \actson \Prob(K)$ and $\Gamma \actson C(K)$. The latter is an action by automorphisms on the C*-algebra of continuous functions on $K$.

The following elementary lemma will be needed later on.

\begin{lem}\label{diffuse measure}
A compact Hausdorff space carries a diffuse Borel probability measure if and only if it contains a perfect set.
\end{lem}
\begin{proof}
The support of a diffuse measure is obviously a perfect set. Conversely, assume that $X$ contains a perfect set $Y$. Then after replacing $X$ by $\bar{Y}$, we may assume that $X$ itself is perfect. Observe that any perfect compact Hausdorff space contains two non-empty disjoint closed subsets which are again perfect. 
Indeed, this can be deduced from the following fact: if $U$ is open subset of a perfect set $X$, then $\bar{U}$ is perfect.
Therefore we may construct by induction a family of closed subsets $K_{i,j}$, $i \geq 0$, $1 \leq j \leq 2^i$ such that for every $i$, $K_{i,j_1}$ and $K_{i,j_2}$ are disjoint whenever $j_1 \neq j_2$ and such that $K_{i+1,2j} \cup K_{i+1,2j+1} \subset K_{i,j}$ for all $i \geq 0$, $1 \leq j \leq 2^i$.
For all $i$, we denote by 
\[K_i := \bigcup_{j = 1}^{2^i} K_{i,j} \text{ and } K := \bigcap_{i} K_i.\]
Since $K_i$ is a decreasing sequence of compact sets, $K$ is non-empty. Define a map $\pi: K \to \{0,1\}^\N$ by the formula $\pi(x)_i := j \mod 2$, where $j$ is the unique index such that $x \in K_{i,j}$. By the construction of $K$ this map is onto and continuous. So we may pull back any diffuse measure $\nu$ on $\{0,1\}^\N$ to a measure on $K$, which will  also be diffuse. Specifically, we define a state $\psi:C(\{0,1\}^\N)\rightarrow\mathbb C$  by $\psi(f)=\int f\;\text{d}\nu$. Viewing $C(\{0,1\}^\N)$ as a C$^*$-subalgebra of $C(K)$, we extend $\psi$ to a state $\varphi:C(K)\rightarrow\mathbb C$. The measure $\mu$ on $K$ given by $\varphi(f)=\int f\;\text{d}\mu$, for every $f\in C(K)$, satisfies $\pi_*\mu=\nu$. Since $\nu$ is diffuse, $\mu$ is also diffuse. 
\end{proof}

We will often consider unitary representations $(\pi,H)$ of $\Gamma$. We will denote by $(\opi,\oH)$ the conjugate representation, that is, $\oH$ is the complex conjugate Hilbert space of $H$ and the representation $\opi$ is such that $\opi(g)(\oxi) = \overline{\pi(g)(\xi)}$ for all $g \in \Gamma$, $\xi \in H$.
Also following \cite{Be90} we say that a representation $(\pi,H)$ is amenable if there exists an $\Ad(\pi(\Gamma))$-invariant state on $B(H)$. This is equivalent to the representation $\pi \ot \opi$ having almost invariant vectors \cite[Theorem 5.1]{Be90}.

\subsection*{Von Neumann algebras} A tracial von Neumann algebra is a pair $(Q,\tau)$  consisting of a von Neumann algebra $Q$ and a tracial faithful normal trace $\tau:Q\rightarrow\mathbb C$. We consider the corresponding 2-norm given by $\Vert x \Vert_2=\sqrt{\tau(x^*x)}$, $x \in M$ and denote by $L^2(Q)$ the associated GNS Hilbert space.  If $P \subset Q$ is a von Neumann subalgebra, we denote by $E_P: Q \to P$ the trace-preserving conditional expectation and by $e_P \in B(L^2(Q))$ the orthogonal projection onto $L^2(P) \subset L^2(Q)$. Jones' basic construction is denoted by $\langle Q,e_P\rangle$.

We write $\mathcal U(Q)$ the group of unitaries of $Q$. For any set $S\subset Q$  is which closed under the adjoint $*$-operation, we denote by $S\dpr$ its double commutant, which, by von Neumann's double commutant theorem is precisely the von Neumann algebra generated by $S$.

The von Neumann algebra of a discrete group $\Gamma$ is denoted by $L\Gamma = \lambda(\Gamma)\dpr\subset B(\ell^2(\Gamma))$. We endow $L\Gamma$ with the canonical trace $\tau: x \mapsto \langle x\delta_e,\delta_e\rangle$, making it a tracial von Neumann algebra. The canonical unitaries $\lambda(g) \in L\Gamma$ will be often denoted by $u_g$, $g \in \Gamma$.

Inside $B(\ell^2(\Gamma))$, we also consider the abelian von Neumann algebra $\ell^\infty(\Gamma)$, acting by pointwise multiplication.

Given a representation $(\pi,H)$ of $\Gamma$ we will consider the corresponding $L\Gamma$-bimodule, whose underlying Hilbert space is  $H \ot \ell^2(\Gamma)$ and the left and right actions are characterized by the formula
\[ u_s \cdot  (\xi \ot \delta_g) \cdot u_t = \pi(s)(\xi) \ot \delta_{sgt}, \text{ for all } g,s,t \in \Gamma, \xi \in H.\]

We recall two notions which will play a key roles in our main results and their proofs: weak compactness and Popa's intertwining-by-bimodules.

\begin{defn}\cite{OP10a}\label{WC}
A trace preserving action $\sigma:\Gamma \actson (Q,\tau)$ on a tracial von Neumann algebra $(Q,\tau)$ is \emph{weakly compact} if there exists a state $\varphi$ on $B(L^2(Q))$ such that $\varphi_{|Q}=\tau$ and $\varphi\circ\text{Ad}(u)=\varphi$, for every $u\in\mathcal U(Q)\cup\sigma(\Gamma)$.
A regular inclusion of tracial von Neumann algebras $Q \subset M$ is weakly compact if  the action $\cN_M(Q) \actson Q$ is weakly compact, where $\cN_M(Q)=\{u\in\mathcal U(M)\mid uQu^*=Q\}$ denotes the normalizer of $Q$ in $M$.
\end{defn}

The following definition/theorem is due to Popa \cite{Po01,Po03}. 

\begin{defn}\label{intertwining}
Consider a tracial von Neumann algebra $M$ with two von Neumann subalgebras $P, Q \subset M$. We say that a corner of $P$ embeds into $Q$ inside $M$, and write $P \prec_M Q$ if one of the following equivalent statements holds.
\begin{enumerate}[(i)]
\item There exists projections $p \in P$, $q \in Q$, a $*$-homomorphism $\varphi: pPp \to qQq$ and a non-zero element $v \in qMp$ such that $\varphi(x)v = vx$ for all $x \in pPp$;
\item There exists no net of unitaries  $(u_n) \subset \cU(P)$ such that $\lim_n \Vert E_Q(xu_ny)\Vert_2 = 0$ for all $x,y \in M$;
\end{enumerate}
\end{defn}

In practice we will use the following characterization, which comes from \cite[Lemma 3.3]{OP10b} (see also \cite[Theorem 2]{BH16} for a general version).

\begin{lem}\label{intertwining}
Consider a tracial von Neumann algebra $M$ with two von Neumann subalgebras $P, Q \subset M$. Then $P \prec_M Q$ if and only if there exists a $P$-central state $\Phi: \langle M,e_Q\rangle \to \C$ which is normal on $M$ and does not vanish on $Me_QM$.
\end{lem}

%%%%%%%%%%
\section{Boundary pieces of  discrete groups}

\begin{defn}\label{defn:piece}
Given a discrete group $\Gamma$, a {\it boundary piece} is a closed subset $X \subset \partial\Gamma$ which is invariant under the left and right $\Gamma$-actions.
\end{defn}

\begin{note}\label{nota1}
Given a discrete group $\Gamma$ and a boundary piece $X \subset \partial \Gamma$ we define the ideal $I_0(X) \subset \ell^\infty(\Gamma) \simeq C(\Delta\Gamma)$ consisting of functions that vanish on $X$. We also denote by $p_X \in {\ell^\infty(\Gamma)}^{**}$ the support projection of this ideal, namely $p_X$ is the unit of ${I_0(X)}^{**}$ inside ${\ell^\infty(\Gamma)}^{**}$. For convenience we will denote by $q_X = 1-p_X$, so that $q_X \ell^\infty(\Gamma) \simeq \ell^\infty(\Gamma)/I_0(X) \simeq C(X)$. It is easily seen that $q_X = \mathbf 1_X$, the indicator function.
\end{note}

Since $X$ is left and right $\Gamma$-invariant, so is $I_0(X)$ and hence $\lambda_g q_X\lambda_g^* = \rho_g q_X\rho_g^* = q_X$ for all $g \in \Gamma$. These equalities are meant inside ${B(\ell^2\Gamma)}^{**}$ (which contains ${\ell^\infty(\Gamma)}^{**}$).

Let us give a first example of a boundary piece, taken from \cite[Section 15.1]{BO08}.

\begin{examp}\label{relative piece}
Take a discrete group $\Gamma$ and a family $\cG$ of subgroups of $\Gamma$. A subset $\Omega \subset \Gamma$ is {\it small relative} to $\cG$ if it is contained in a finite union of sets of the form $s\Lambda t$ with $s,t \in \Gamma$, $\Lambda \in \cG$. The C*-subalgebra $c_0(\Gamma,\cG) \subset \ell^\infty(\Gamma)$ generated by functions whose support is a small set in $\Gamma$ is an ideal in $\ell^\infty(\Gamma)$ which is globally left and right $\Gamma$-invariant and contains $c_0(\Gamma) $. In particular, it is of the form $I(X)$ for some boundary piece $X = X(\cG) \subset \partial\Gamma$.
Alternatively, $X(\cG)$ is described as the intersection of all $\overline{\Delta\Gamma \setminus \Omega}$, where $\Omega$ ranges over all sets that are small relative to $\cG$.
\end{examp}

In \cite[Section 15]{BO08}, the ideal $c_0(\Gamma,\cG)$ is used to construct an ideal of ``relatively compact" operators of $B(\ell^2(\Gamma))$. % shown to be related to an ideal of compact operators. 
We will do the same here for general boundary pieces. This will allow us later to generalize the notion of relative bi-exactness to our setting.

%%%%%
\subsection{Compact operators towards a boundary piece}\label{compactt}

For a discrete group $\Gamma$ and a set $U \subset \Delta\Gamma$, we denote by $P_U:\ell^2(\Gamma) \to \ell^2(\Gamma)$ the orthogonal projection onto the subspace $\ell^2(\Gamma \cap U)$, that is, $P_U(\delta_g) = \mathbf{1}_{g \in U} \delta_g$ for all $g \in \Gamma$.

\begin{defn}\label{mass}
Fix a discrete group $\Gamma$, a closed set $X \subset \partial \Gamma$ and a bounded net of vectors $\xi_n \in \ell^2(\Gamma)$. We say that $(\xi_n)_n$ has
\begin{itemize}
\item {\it positive mass on $X$} if there exists $\eps > 0$ such that for any neighborhood $U$ of $X$ inside $\Delta\Gamma$, we have $\Vert P_U(\xi_n) \Vert > \eps$ for all $n$ large enough;
\item {\it full mass on $X$} if for any neighborhood $U$ of $X$ inside $\Delta \Gamma$,  we have 
\[\limsup_n \Vert \xi_n - P_U(\xi_n)\Vert = 0.\]
\end{itemize}
\end{defn}

\begin{defn}\label{compact}
In the above setting, an operator $T \in B(\ell^2\Gamma)$ is said to be {\it compact towards $X$} if for any bounded net of vectors $\xi_n \in \ell^2\Gamma$ with full mass on $X$, we have $\lim_n \Vert T\xi_n\Vert = 0$. We denote by $\K(\Gamma;X)$ the set of all operators that are compact towards $X$, and note that it is a hereditary C*-subalgebra of $B(\ell^2\Gamma)$.
\end{defn}

Fix a group $\Gamma$ and a boundary piece $X \subset \partial\Gamma$. Recall that $I_0(X) \subset \ell^\infty(\Gamma)$ denotes the ideal of functions on $\Delta\Gamma$ that vanish on $X$.
Consider the following C*-algebra acting on $\ell^2\Gamma$:
\[A_\Gamma := C^*(\ell^\infty(\Gamma),\lambda(\Gamma),\rho(\Gamma)) \subset B(\ell^2\Gamma).\]
Denote by $I(X) \subset A_\Gamma$ the ideal generated by $I_0(X)$. Since $c_0(\Gamma)\subset I_0(X)$, we get that $I(X)$ contains the ideal of compact operators, by irreducibility. In fact, we have the following characterization.

\begin{lem}\label{au}
Any approximate unit $(e_i)_{i \in I}$ of $I_0(X)$ is an approximate unit for $I(X)$. In particular, \[I(X) = A_\Gamma \cap \K(\Gamma;X).\]
\end{lem}
\begin{proof}
Recall that $I_0(X)$ is by definition the set of continuous functions on $\Delta\Gamma$ which vanish on $X$. Since $X$ is left and right $\Gamma$-invariant, it is clear that $\lambda_g I_0(X) \lambda_g^* \subset I_0(X)$ and $\rho_g I_0(X) \rho_g^* \subset I_0(X)$ for all $g \in \Gamma$. In particular, the ideal $I(X)$ may be described as the norm closure of the linear span of $\{\lambda_g \rho_h f \, , \, f \in I_0(X), g,h \in \Gamma\}$, or alternatively, as the norm closure of the linear span of $\{f\lambda_g \rho_h  \, , \, f \in I_0(X), g,h \in \Gamma\}$. With these two descriptions it is now clear that any approximate unit for $I_0(X)$ is indeed an approximate unit for $I(X)$.

Let us prove the second part of the statement. Clearly, $A_\Gamma \cap \K(\Gamma;X)$ is an ideal inside $A_\Gamma$ and contains $I_0(X)$. So $I(X) \subset A_\Gamma \cap \K(\Gamma;X)$. Conversely, let $T \in A_\Gamma \cap \K(\Gamma;X)$. 
Let $\mathcal N$ be the set of open neighbourhoods of $X$ in $\Delta\Gamma$, ordered by inverse inclusion. For $U\in\mathcal N$, let $e_U\in C(\Delta\Gamma)$ such that $0\leq e_U\leq 1$, $e_U\equiv 0$ on $X$, and $e_U\equiv 1$ on $\Delta\Gamma\setminus U$. Then $(e_U)_{U\in\mathcal N}$ is an approximate unit of $I_0(X)$, and for any net of unit vectors $(\xi_U)_U$, the net $((1- e_U)\xi_U)_U$ has full mass on $X$. Since $T$ is compact towards $X$, $\lim_U \Vert T(1-e_U)\xi_U\Vert = 0$. This is easily seen to imply $\lim_U Te_U = T$. Hence $T \in I(X)$, by the first part of the lemma. \end{proof}

%and take an approximate unit of $I_0(X)$ made of functions $e_i$ on $\Delta\Gamma$ that vanish on $X$ but are equal to $1$ outside smaller and smaller neighborhoods of $X$. Then for any net of unit vectors $\xi_i$, the net $((1- e_i)\xi_i)_i$ has full mass on $X$. Since $T$ is compact towards $X$, $\lim_i \Vert T(1-e_i)\xi_i\Vert = 0$. This is easily seen to imply $\lim_i Te_i = T$. Hence $T \in I(X)$, by the first part of the lemma.

Let us now provide several classes of examples of boundary pieces arising from various kinds of data: geometric, cohomological, and representation theoretic.

%%%%%
\subsection{Proximal pieces}

Consider a discrete group $\Gamma$ with a continuous action on a compact space $\Gamma \actson K$ and take a probability measure $\eta \in \Prob(K)$.

\begin{defn}
A point $\omega \in \Delta\Gamma$ (i.e. an ultrafilter on $\Gamma$) is called an {\it $\eta$-proximal} element if for all $h \in \Gamma$, we have $\lim_{g \to \omega} ((gh)\cdot \eta - g\cdot \eta ) = 0$, in the weak-* topology. We denote by $\partial_\eta\Gamma \subset \Delta \Gamma$ the set of $\eta$-proximal elements. 
\end{defn}

The term proximal refers to the fact that the action of $\Gamma$ on $\Prob(K)$ pushes the whole orbit $\Gamma \eta$ to a single point (namely $\lim_{g \to \omega} g\eta$) when going in the direction of $\omega$.

\begin{lem}\label{first facts}
The set $\partial_\eta\Gamma$ enjoys the following properties.
\begin{enumerate}
\item It is a closed subset of $\Delta\Gamma$.
\item It is left and right $\Gamma$-invariant.
\item If $\eta$ is $\Gamma$-invariant, then $\partial_\eta \Gamma = \Delta \Gamma$. Otherwise, $\partial_\eta \Gamma \subset \partial\Gamma$.
\end{enumerate}
\end{lem}
\begin{proof}
Consider the orbit map $g \in \Gamma \mapsto g\cdot \eta \in \Prob(K)$. Since $\Prob(K)$ is compact for the weak-* topology, this map extends to a continuous map $\sigma: \Delta\Gamma \to \Prob(K)$. Then $\partial_\eta\Gamma$ can be expressed as $\bigcap_{h \in \Gamma} \{\omega \, : \, \sigma(\omega h) = \sigma(\omega)\}$, which implies $(1)$. Moreover, if $\omega \in \partial_\eta\Gamma $ and $g \in \Gamma$ then for all $h \in \Gamma$, we have $\sigma(\omega gh) = \sigma(\omega) = \sigma(\omega g)$. So $\partial_\eta\Gamma$ is invariant under the right action. The left invariance follows from the fact that $\sigma(g\omega) = g\sigma(\omega)$ for all $g \in \Gamma$. So $(2)$ holds true. Statement $(3)$ is obvious.
\end{proof}

We will  refer to the boundary pieces of the form $\partial_\eta\Gamma$ as {\it proximal pieces}.
A key feature of this construction of boundary pieces is that properties of the initial action $\Gamma \actson K$ (for example, the existence of invariant measures or amenability) can be transferred to properties of the corresponding boundary piece. This is based on the following lemma.

\begin{lem}\label{lem: theta}
Assume that the measure $\eta$ is not $\Gamma$-invariant, so that $X := \partial_\eta\Gamma$ is indeed a boundary piece. Denote by $\lambda$ and $\rho$ the left and right actions of $\Gamma$ on $C(X)$, respectively. Then there exists a unital completely positive map $\theta: C(K) \to C(X)$ such that for all $g \in \Gamma$, $f \in C(K)$, we have
\[\theta(g . f) = \lambda_g(\theta(f)) \text{ and } \rho_g(\theta(f)) = \theta(f).\]
\end{lem}
\begin{proof}
For $f \in C(K)$ and $g \in \Gamma$, define $\theta_0(f)(g) := \int_K f d(g. \eta)$. The map $\theta_0: C(K) \to \ell^\infty(\Gamma)$ obtained this way is unital and completely positive. Denote by $\theta$ the composition of $\theta_0$ with the restriction map $\ell^\infty(\Gamma) \simeq C(\Delta\Gamma) \to C(X)$. The following two computations imply the lemma: for all $f \in C(K)$, $\omega \in X = \partial_\eta\Gamma$ and $h \in \Gamma$, we have
\[\theta(h \cdot f)(\omega) = \lim_{g \to \omega} \theta_0(h \cdot f)(g) = \lim_{g \to \omega} \theta_0(f)(h^{-1}g) = \theta(f)(h^{-1}\omega),\]
\[\theta(f)(\omega h) = \lim_{g \to \omega} \theta(f)(gh) =  \lim_{g \to \omega} \int_K f d((gh)\cdot \eta) =  \lim_{g \to \omega} \int_K f d(g \cdot \eta) = \theta(f)(\omega).\qedhere\]
\end{proof}

\begin{lem}\label{transfer}
Assume that $\eta$ is not $\Gamma$ invariant and write $X := \partial_\eta\Gamma$. Denote by $C(X)^{\Gamma_r}$ the C*-algebra of right-$\Gamma$-invariant continuous functions on $X$. We have the following two facts:
\begin{itemize}
\item If there is no $\Gamma$-invariant probability measure on $K$, then there is no left-$\Gamma$-invariant state on $C(X)^{\Gamma_r}$,
\item If the action $\Gamma \actson K$ is topologically amenable, then the left $\Gamma$- action $\Gamma \actson C(X)^{\Gamma_r}$ is amenable. We refer to \cite[Section 4.3]{BO08} for the definition of amenable actions.
\end{itemize}
\end{lem}
\begin{proof}
The map $\theta: C(K) \to C(X)$ given by Lemma \ref{lem: theta} ranges into $C(X)^{\Gamma_r}$ and is left-$\Gamma$-equivariant. So if there exists a left-$\Gamma$-invariant state $\varphi$ on $C(X)^{\Gamma_r}$, then $\varphi \circ \theta$ is a $\Gamma$-invariant state on $C(K)$, so there exists a $\Gamma$-invariant probability measure on $K$. This proves the first fact.

Note that $\theta$ induces a $\Gamma$-equivariant continuous map from the Gelfand spectrum $\tilde{X}$ of $C(X)^{\Gamma_r}$ to the state space $\cS(C(K))$, which is nothing but $\Prob(K)$. If the action of $\Gamma$ on $K$ is topologically amenable then so is the action on $\tilde{X}$. Hence $\Gamma \actson C(X)^{\Gamma_r}$ is amenable.
\end{proof}

The next lemma shows that every boundary piece is a proximal piece, and the involved action can be chosen so that the converse of the previous lemma holds true.

\begin{lem}\label{all proximal}
For every boundary piece $X \subset \partial\Gamma$ there exists a continuous action of $\Gamma$ on a compact space $K$, with a probability measure $\eta \in \Prob(K)$ such that $X = \partial_\eta \Gamma$. Moreover the action $\Gamma \actson K$ may be chosen so that:
\begin{itemize}
\item If there is no left $\Gamma$-invariant state on $C(X)^{\Gamma_r}$ then there is no $\Gamma$-invariant probability measure on $K$;
\item If the action $\Gamma \actson C(X)^{\Gamma_r}$ is amenable then $\Gamma \actson K$ is topologically amenable.
\end{itemize}
\end{lem}
\begin{proof}
Denote by $\phi: \ell^\infty(\Gamma) \simeq C(\Delta\Gamma) \to C(X)$ the restriction map, by $A := C(X)^{\Gamma_r} \subset C(X)$ and by $B := \phi^{-1}(A) \subset \ell^\infty(\Gamma)$. Note that $\phi$ is $\Gamma \times \Gamma$-equivariant and $B$ is globally $\Gamma \times \Gamma$-invariant. By dualyzing the $\Gamma \times \Gamma$-embeddings $\ker(\phi) \subset B \subset \ell^\infty(\Gamma)$, we find a compact $\Gamma \times \Gamma$-space $K$ and a continuous $\Gamma \times \Gamma$-equivariant map $\pi: \Delta\Gamma \to K$ which restricts to a homeomorphism on $\Delta\Gamma \setminus X$ and such that $B \simeq C(K)$ and the embedding $B \subset \ell^\infty(\Gamma)$ corresponds to the map 
\[f \in C(K) \mapsto f \circ \pi \in \ell^\infty(\Gamma).\]
Since the right $\Gamma$-action on $A$ is trivial, one easily verifies that $\pi(\omega h) = \pi(\omega)$ for all $\omega \in X$, $h \in \Gamma$.

We are interested in the left action of $\Gamma$ on $K$, and the measure $\eta := \delta_{\pi(e)} \in \Prob(K)$, the Dirac mass at the image of the neutral element $e \in \Gamma$. For $\omega \in X$ and $h \in \Gamma$  we have
\[\lim_{g \to \omega} (gh)\cdot \eta = \lim_{g \to \omega} \delta_{\pi(gh)} = \delta_{\pi(\omega h)} = \delta_{\pi(\omega)} = \lim_{g \to \omega} g \cdot \eta.\]
This shows that $X \subset \partial_\eta\Gamma$. Conversely, if $\omega \in \partial\Gamma \setminus X$, we may take $h \in \Gamma$ such that $\omega h \neq \omega$. Then since $\pi_{\vert \Delta\Gamma \setminus K}$ is a homeomorphism, we have $\pi(\omega h) \neq \pi(\omega)$. The above computation then shows that $\lim_{g \to \omega} (gh)\cdot \eta \neq \lim_{g \to \omega} g \cdot \eta$. Hence $\omega \notin \partial_\eta\Gamma$. This proves that $X = \partial_\eta\Gamma$.

Now let us verify the two properties of this action.
If there is no left-$\Gamma$-invariant state on $C(X)^{\Gamma_r}$, then in particular, $\Gamma$ is non-amenable.
If $\mu$ is an invariant probability measure on $K$, then $\mu$ has zero mass on $\pi(\Delta \Gamma \setminus X)$, because $\Gamma$ is non-amenable and the restriction of $\pi$ to $\Delta \Gamma \setminus X$ is a $\Gamma$-equivariant homeomorphism. In particular, the state on $B = C(K)$ given by integration with respect to $\mu$ vanishes on $\ker(\phi)$. It thus induces a $\Gamma$-invariant state on $A$, which was excluded. We conclude that there is no such measure $\mu$.

Let us now assume that the action $\Gamma \actson C(X)^{\Gamma_r}$ is amenable, and show that the action on $K$ is amenable. By \cite[Theorem 4.4.3]{BO08}, it suffices to check that the reduced crossed product by the left action $B \rtimes_r \Gamma$ is nuclear. We will show that it is an extension of nuclear C*-algebras. 

Recall that $I_0(X) \subset \ell^\infty(\Gamma)$ denotes the ideal corresponding to the boundary piece $X$, i.e. $I_0(X) = C_0(\Delta\Gamma \setminus X)$. It is left and right $\Gamma$-invariant and hence the reduced crossed product by the left $\Gamma$-action $I_0(X) \rtimes_r \Gamma$ is an ideal inside the uniform Roe algebra $\ell^\infty(\Gamma)\rtimes_r \Gamma$. Since $\Gamma$ is exact \cite[Theorem 5.1.7]{BO08}, the uniform Roe algebra is nuclear and hence $I_0(X) \rtimes_r \Gamma$ is nuclear as well.

On the other hand we have an exact sequence $0 \to I_0(X) \to B \to A \to 0$ of left $\Gamma$ C*-algebras. Using again the fact that $\Gamma$ is exact, we get that the following sequence is exact
\[0 \to I_0(X) \rtimes_r \Gamma \to B \rtimes_r \Gamma \to A \rtimes_r \Gamma \to 0.\]
Since the action of $\Gamma$ on $A$ is amenable, the C*-algebra $ A \rtimes_r \Gamma$ is nuclear. We conclude from \cite[Proposition 10.1.3]{BO08} that $B \rtimes_r \Gamma$ is nuclear, as wanted.
\end{proof}

%%%%%
\subsection{Boundary pieces from arrays}

Consider a discrete group $\Gamma$ and a unitary representation $\pi: \Gamma \to \cU(H)$ into a Hilbert space $H$. We recall the following definition from \cite[Definition 1.1, Proposition 1.5]{CS13}. %The following definition is a variation of \cite[Definition 1.1]{CS13} and \cite[Definition 15.1.2]{BO08}.

\begin{defn}
A {\it two-sided array} of $\Gamma$ into $\pi$ is a map $b: \Gamma \to H$ such that 
\[\sup_{g \in \Gamma} \Vert b(sgt) - \pi_s(b(g))\Vert < \infty, \text{ for all } s,t \in \Gamma.\]
\end{defn}

Of course this notion is interesting only if $b$ is unbounded. If this is the case, we can construct a boundary piece.

\begin{lem}\label{array piece}
If $b: \Gamma \to H$ is an unbounded array then $X := \{\omega \in \Delta\Gamma \mid \lim_{g \to \omega} \Vert b(g) \Vert = +\infty\}$ is a non-empty boundary piece. Moreover there exists a unital completely positive map $\theta: B(H) \to C(X)$ such that
\[ \theta \circ \Ad(\pi(g)) = \lambda_g \circ \theta \text{ and } \rho_g \circ \theta = \theta, \text{ for all } g \in \Gamma.\]
\end{lem}
\begin{proof}
The set $X$ is obviously closed and contained in $\partial \Gamma$. The left and right invariance of $X$ follow from the fact that $b$ is a two sided array. 

For all $T \in B(H)$ and $g \in \Gamma$, set $\theta_0(T)(g) := \langle T \xi_g ,\xi_g \rangle$, with $\xi_g := \frac{b(g)}{\Vert b(g) \Vert}$. The resulting map $\theta_0: B(H) \to \ell^\infty(\Gamma)$ is unital and completely positive. Define $\theta$ as the composition of $\theta_0$ with the restriction map $\ell^\infty(\Gamma) \to C(X)$.
As in the proof of Lemma \ref{lem: theta}, the desired properties of $\theta$ easily follow from the fact that $\lim_{g \to \omega} \Vert \pi_s(\xi_g) - \xi_{sgt}\Vert = 0$, for all $\omega \in X$, and all $s,t \in \Gamma$.
\end{proof}

%Note that the existence of an unbounded two-sided array into a given representation $(\pi,H)$ is a non-trivial cohomological condition. 
As  emphasized in \cite{CS13}, cocycles and quasi-cocycles are examples of two-sided arrays.
 %According to \cite{CS13}, a group which satisfies this condition is said to belong to the class $\cQ\cH$ (or $\cQ\cH(\pi)$ if one wants to specify the representation $\pi$). As Chifan and Sinclair emphasize, cocycles and quasi-cocycles are examples of two sided arrays.

%%%%%
\subsection{Mixing pieces}

Although we won't need it in the sequel, we present here a very natural example of boundary piece, which is of independent interest.

Let $\pi: \Gamma \to \cU(H)$ be a unitary representation. Recall that $\pi: \Gamma \to \cU(H)$ is {\it weakly mixing} if $0$ is a weak operator topology accumulation point of $\pi(\Gamma)$, while $\pi$ is {\it mixing} if this is the unique accumulation point. The unit ball $(B(H) )_1$ is compact with respect to the weak operator topology and hence we may extend $\pi$ continuously to a map from $\Delta \Gamma$ into $(B(H))_1$. The set $X(\pi) = \pi^{-1}( \{ 0 \} ) \subset \partial\Gamma$ is easily seen to be a boundary piece which records all the directions in which $\pi$ is mixing. We call it the {\it mixing piece} of $\pi$.

\begin{lem}\label{lem:wotnbhd}
Let $\pi: \Gamma \to \cU(H)$ be a representation. Let $U$ be a neighborhood of $X(\pi)$. Then there exists a weak operator topology neighborhood $O$ of $0 \in B(H)$ such that $\pi^{-1}(O) \subset U$. 
\end{lem}
\begin{proof}
For each point $\omega \in \Delta \Gamma \setminus U$ there exists a WOT-neighborhood $O$ of $0$ in $(B(H))_1$ such that $\pi(\omega) \not\in \overline{O}$, and hence $\omega$ is not in the closure of $\pi^{-1}(O)$. By compactness of $\Delta \Gamma \setminus U$ it then follows that there is a WOT-neighborhood $\tilde O$ such that $\pi^{-1}(\tilde O) \subset U$. 
\end{proof}

\begin{lem}\label{lem:mixingsequence}
Let $\pi: \Gamma \to \cU(H)$ be a representation. Let $(c_n)$ be a uniformly bounded net in $L\Gamma$. The following conditions are equivalent:
\begin{enumerate}[(i)]
\item $(c_n)$ has full mass on $X(\pi)$\footnote{Here we view $(c_n)$ as a bounded net in $\ell^2(\Gamma)$ so Definition \ref{mass} makes sense.}.
\item For all other representation $(\rho,K)$, all $\xi, \eta \in H \ot K \ot \ell^2(\Gamma)$ and all uniformly bounded net $(d_n) \subset L\Gamma$ we have $\lim_{n \to \infty} \langle c_n \xi d_n, \eta \rangle = 0$.
\item $\lim_{n \to \infty} \langle c_n ( \xi \otimes \delta_e ) c_n^*, \eta \otimes \delta_e \rangle = 0$ for all $\xi, \eta \in H \ot \overline{H}$.
\item $\lim_{n \to \infty} \langle c_n ( \xi \otimes \delta_e ) c_n^*, \eta \otimes \delta_e \rangle = 0$ for all vectors $\xi, \eta \in H \ot \overline{H}$ of the form $\xi = \xi_0 \ot \overline {\xi_0}$, $\eta = \eta_0 \ot \overline {\eta_0}$, with $\xi_0,\eta_0 \in H$.
\end{enumerate}
\end{lem}
\begin{proof}
Suppose $(i)$ holds. Note that for any representation $(\rho,K)$ we have $X(\pi) \subset X(\pi \ot \rho)$. So (i) implies that $(c_n)$ has full mass on $X(\pi \ot \rho)$. So in order to prove $(ii)$ we may as well replace $\pi$ with $\pi \ot \rho$ and assume that $(\rho,K)$ is trivial.
As $(c_n)$ and $(d_n)$ are uniformly bounded it suffices for $(2)$ to consider $\xi$ and $\eta$ of the form $\xi = \xi_0 \otimes \delta_e$, $\eta = \eta_0 \otimes \delta_e$ with $\xi_0, \eta_0 \in H$. Take $\eps > 0$ and let $O = \{ T \in (B(H) )_1 \mid | \langle T \xi_0, \eta_0 \rangle | < \varepsilon \}$. We then compute as in Lemma 2.5 of \cite{chifanpeterson}
\[| \langle c_n \xi d_n, \eta \rangle |
\leq \| \xi_0 \| \| \eta_0 \| \| d_n \|_2 \| P_{\pi^{-1}(O)^c}(c_n) \|_2 + \varepsilon \| c_n \|_2 \| d_n \|_2.\]
As $(c_n)$ has full mass on $X$ and as $\varepsilon > 0$ was arbitrary condition $(ii)$ then follows.

The implications $(ii) \Rightarrow (iii) \Rightarrow (iv)$ are obvious. Suppose now that $(iv)$ holds. Fix $\varepsilon > 0$ and take a neighborhood $U$ of $X(\pi)$. By Lemma~\ref{lem:wotnbhd} there is a weak operator topology neighborhood $O$ of $0$ so that $\pi^{-1}(O) \subset U$. We may take $O$ of the form 
\[O = \{ T \in (B(H))_1 \mid \sum_{i = 1}^k | \langle T \xi_i, \eta_i \rangle |^2 < \varepsilon \}.\]
Then from $(iv)$ we have $\lim_{n \to \infty} \sum_{i = 1}^k  \langle c_n (\xi_i \ot \overline{\xi_i} \otimes \delta_e) c_n^*, \eta_i \ot \overline{\eta_i} \otimes \delta_e \rangle = 0$, whereas, writing $c_n = \sum_{g \in \Gamma} c_{n,g}u_g$ for the Fourier expansion, we also have
\begin{align*}
\lim_{n \to \infty} \sum_{i = 1}^k  \langle c_n (\xi_i \ot \overline{\xi_i} \otimes \delta_e) c_n^*, \eta_i \ot \overline{\eta_i} \otimes \delta_e \rangle & = \lim_n \sum_{i = 1}^k \sum_{g \in \Gamma} \vert c_{n,g}\vert^2\vert \langle \pi(g) \xi_i, \eta_i \rangle \vert^2\\
& \geq \varepsilon \limsup_{n \to \infty} \| P_{\pi^{-1}(O)^c}(c_n) \|_2^2\\
& \geq \varepsilon \limsup_{n \to \infty} \| P_{U^c}(c_n) \|_2^2.
\end{align*}
Thus, $(c_n)$ has full mass on $X(\pi)$.
\end{proof}

The gain in $(iv)$ above compared to $(iii)$ is that the inner product $\langle u ( \xi \otimes \delta_e ) u^*, \eta \otimes \delta_e \rangle$ becomes non-negative for all $u$ as soon as $\xi$ and $\eta$ are as in $(iv)$. This will be used in the proof of the following standard weak mixing/compact dichotomy (see, e.g.\  \cite{Po01}).

\begin{prop}[Weak mixing/compact dichotomy]
Let $B \subset L\Gamma$ be a von Neumann subalgebra, and $\mathcal G \subset \mathcal U(B)$ a group which generates $B$ as a von Neumann algebra.  The following are equivalent:
\begin{enumerate}[(i)]
\item Some net of unitaries $(u_n) \subset \mathcal G$ has full mass on $X(\pi)$.
\item The $L\Gamma$-bimodule $H \ot \overline{H} \otimes \ell^2 \Gamma$ has no non-zero $B$-central vectors. 
\end{enumerate}
\end{prop}
\begin{proof}
If some net of unitaries $(u_n) \subset \mathcal G$ has full mass on $X(\pi)$ then from condition $(ii)$ of Lemma~\ref{lem:mixingsequence} we see that $H \ot \oH \otimes \ell^2 \Gamma$ can have no non-zero $B$-central vectors. 

Conversely, assume that there is no net of unitaries $(u_n) \subset \mathcal G$ which has full mass on $X(\pi)$. 

{\bf  Claim.} There exist $\eps >0$ and vectors $\xi_1, \dots, \xi_k \in H$, $\eta_1,\dots,\eta_k \in H$ such that 
\[ \sum_{i = 1}^k  \langle u (\xi_i \ot \overline{\xi_i} \otimes \delta_e) u^*, \eta_i \ot \overline{\eta_i} \otimes \delta_e \rangle \geq \eps, \text{ for all } u \in \cG.\]

Indeed since each term in the above sum is non-negative, if we find $u \in \cG$ such that the reverse inequality holds, then $ \langle u (\xi_i \ot \overline{\xi_i} \otimes \delta_e) u^*, \eta_i \ot \overline{\eta_i} \otimes \delta_e \rangle < \eps$ for all $i$. So if the claim did not hold we could easily construct a net $(c_n)$ of elements of $\cG$ satisfiying $(iv)$ of Lemma \ref{lem:mixingsequence}. This would contradict our assumption.

Now, denote by $\xi := \oplus_{i=1}^k \xi_i \ot \overline{\xi_i}$ and $\eta := \oplus_{i = 1}^k \eta_i \ot \overline{\eta_i}$ in $(H \ot \oH)^{\oplus k}$. The claim amounts to
\[\langle u (\xi \otimes \delta_e) u^*, \eta \otimes \delta_e \rangle \geq \eps, \text{ for all } u \in \cG.\]
Denote by $\txi \in (H \otimes \oH)^{\oplus k} \ot \ell^2(\Gamma)$ the unique element of minimal norm in the convex closure of $(\{ u (\xi \otimes \delta_e) u^*\, , \, u \in \cG \})$. Then $\txi$ is $B$-central (since $\mathcal G$ generates $B$ as a von Neumann algebra) and satisfies $\langle \txi, \eta \ot \delta_e \rangle \geq \varepsilon$, hence $\txi \not= 0$. Identifying $(H \otimes \oH)^{\oplus k} \ot \ell^2(\Gamma)$ with $(H \otimes \oH \ot \ell^2(\Gamma))^{\oplus k}$, we find that some coordinate of $\txi$ is a non-zero $B$-central vector inside $H \otimes \oH \ot \ell^2(\Gamma)$.
\end{proof}

%%%%%%%%%%
\section{Properly proximal groups and Cartan subalgebras}

In this section, we exploit proximal pieces of a group $\Gamma$ to study its von Neumann algebra $L\Gamma$. This relies on the following definition.

\begin{defn}\label{patched cv}
We say that a discrete group $\Gamma$ is a {\it properly proximal group} if it admits a finite family of continuous actions on compact spaces $\Gamma \actson K_i$, $i \geq 0$, and probability measures $\eta_i \in \Prob(K_i)$, $i \geq 0$, such that:
\begin{itemize}
\item For all $i$, there is no $\Gamma$-invariant Borel probability measure on $K_i$;
\item $\bigcup_i \partial_{\eta_i}\Gamma = \partial \Gamma$.
\end{itemize}
\end{defn}

Let us mention that amenable groups are never properly proximal, since they obviously never satisfy the first condition above. As we will show, properly proximal groups are in fact never inner amenable (see Proposition \ref{inner amenable}).

\begin{rem}
One may define an, a priori, more general notion of proper proximality by using actions on arbitrary, not necessarily commutative, C$^*$-algebras. However, as we explain below, this leads to the same notion. More precisely, assume that a discrete group $\Gamma$ admits a finite family of continuous actions on C$^*$-algebras $\Gamma \actson A_i$, $i \geq 0$, and states $\varphi_i \in \cS(A_i)$, $i \geq 0$, such that:
\begin{itemize}
\item For all $i$, there is no $\Gamma$-invariant state on $A_i$;
\item $\bigcup_i \partial_{\varphi_i}\Gamma = \partial \Gamma$, where $\partial_{\varphi_i}\Gamma = \{\omega\in\partial\Gamma \mid \text{$\lim_{g\rightarrow\omega}(gh\varphi_i-g\varphi_i)=0$, $^*$-weakly, $\forall h\in\Gamma$}\}$.
\end{itemize}
Define $K_i := \cS(A_i)$ and $\eta_i=\delta_{\varphi_i}\in\text{Prob}(K_i)$, for $i\geq 0$. Then $K_i$ is compact in the weak-$^*$ topology, and the action $\Gamma\curvearrowright K_i$ is continuous. Since $\partial_{\eta_i}\Gamma = \partial_{\varphi_i} \Gamma$, we have $\bigcup_i\partial_{\eta_i}\Gamma = \partial\Gamma$. Moreover, there is no $\Gamma$-invariant probability measure on $K_i$, for all $i$. If $\eta\in\text{Prob}(K_i)$ were $\Gamma$-invariant, then $\varphi=\int_{K_i}\psi\;\text{d}\eta(\psi)$ would be a $\Gamma$-invariant state on $A_i$. This shows that $\Gamma$ is properly proximal.
\end{rem}

%%%%%
\subsection{Equivalent formulations}

The aim of this section is to prove the following result, which asserts among other things that proper proximality can always be observed with a single action. The main implication $(iv) \Rightarrow (iii)$ is due to Narutaka Ozawa. We warmly thank him for allowing us to include his argument here.

\begin{thm}\label{carac pp}
Consider a discrete countable group $\Gamma$, and let $X \subset \partial \Gamma$ be a boundary piece. The following facts are equivalent.
\begin{enumerate}[(i)]
\item There are continuous actions $\Gamma \actson K_i$, $i = 1, \ldots, k$ on compact spaces $K_i$ with probability measures $\eta_i \in {\rm Prob}(K_i)$ such that there is no $\Gamma$-invariant probability measure on any $K_i$ and such that $X \subset \cup_{i = 1}^k \partial_{\eta_i} \Gamma$.
\item There is a single continuous action $\Gamma \actson K$ on a compact space $K$ with a probability measure $\eta \in \Prob(K)$ such that there is no $\Gamma$-invariant probability measure on $K$ and $X = \partial_\eta \Gamma$.
\item There is no left-$\Gamma$-invariant state on $C(X)^{\Gamma_r}$;
\item There is no left-$\Gamma$-invariant state on $(C(X)^{**})^{\Gamma_r}$.
\end{enumerate}
In particular, if $X = \partial\Gamma$, all these conditions are equivalent to proper proximality of $\Gamma$.
\end{thm}
\begin{proof}
$(i) \Rightarrow (iv)$. Consider finitely many actions $\Gamma \actson K_i$ with measures $\eta_i \in \Prob(K_i)$, as in (i). For each $i$, set $X_i := \partial_{\eta_i}\Gamma$. By Lemma \ref{transfer}, we know that for each $i$, there is no left-$\Gamma$-invariant state on $C(X_i)^{\Gamma_r}$. In particular there is no $\Gamma$-invariant state on $(C(X_i)^{**})^{\Gamma_r}$. 

Denote by $p_i \in C(X)^{**}$ the support projection of the ideal $C_0(X \setminus X_i)$ and by $q_i := 1 - p_i$, so that $C(X_i)^{**} = q_iC(X)^{**}$. Since $X \subset \cup_i X_i$, we see that  $\vee q_i = 1$. Moreover, the projections $q_i$ are left and right $\Gamma$-invariant, and in particular, $(C(X_i)^{**})^{\Gamma_r} = q_i(C(X)^{**})^{\Gamma_r}$. If $\varphi$ is a left-$\Gamma$-invariant state on $(C(X)^{**})^{\Gamma_r}$, then there exists some $i$ such that $\varphi(q_i) \neq 0$. Then the restriction of $\varphi$ to $q_i(C(X)^{**})^{\Gamma_r}$ is a non-zero left $\Gamma$-invariant positive linear functional on $(C(X_i)^{**})^{\Gamma_r}$, which contradicts the previous paragraph.

$(iv) \Rightarrow (iii)$. Denote by $A := (C(X)^{**})^{\Gamma_r}$. If there is no $\Gamma$-invariant state on $A$ then there is no non-zero $\Gamma$-invariant linear functional at all. This follows for instance from the uniqueness of the polar decomposition of normal linear functionals on $A^{**}$. It therefore follows from the Hahn-Banach theorem that the linear span of $\{x - g \cdot x \mid x \in A, g \in \Gamma\}$ is norm dense inside $A$. We may thus find $g_1, \dots, g_d \in \Gamma$ and $x_1,\dots,x_d \in A$ such that 
\begin{equation}\label{almost 1}\Vert 1 - \sum_{k = 1}^d (x_k - g_k \cdot x_k)\Vert < 1/2.\end{equation}

The elements $x_k$ belong to $C(X)^{**}$, so we may find for each $k$, a net $(x_k^i)_{i \in I}$ in $C(X)$ which converges to $x_k$ in the weak-* topology of $C(X)^{**}$. By Goldstine's theorem, we may assume that for each $i$ and $k$, we have $\Vert x_k^i\Vert \leq \Vert x_k \Vert$. Since the elements $x_k$ are right-$\Gamma$-invariant, we have the weak-* convergence $\lim_i x_k^i - x_k^i\cdot g = 0$ for all $k$ and $g \in \Gamma$. Thus we may replace the $x_k^i$ by convex combinations to assume that this convergence holds in norm: $\lim_i \Vert x_k^i - x_k^i \cdot g \Vert = 0$, for all $k$ and $g \in \Gamma$. Better, we may further take convex combinations to assume that, in addition,
\[\Vert 1 - \sum_{k = 1}^d (x_k^i - g_k \cdot x_k^i)\Vert < 1/2, \text{ for all } i \in I.\]
This last assertion follows from a classical fact on Banach spaces recorded in Lemma \ref{Banach} below.

The task is now to combine the elements $x_k^i$ to produce elements $z_k \in C(X)$ which are actually right $\Gamma$-invariant and such that \eqref{almost 1} holds with $z_k$'s in the place of $x_k$'s.
Recall that $C(X)$ is naturally identified with $\ell^\infty(\Gamma)/I_0(X)$.

For each $i$, we may take lifts $y_k^i \in \ell^\infty(\Gamma)$ of $x_k^i$, $k = 1, \ldots, d$, such that $\| y_k^i \| \leq \| x_k \|$, and we may also take a lift $b^i \in \ell^\infty(\Gamma)$ of $\sum_{k = 1}^d (x_k^i - g_k \cdot x_k^i)$ so that $\| 1 - b^i \| < 1/2$. 

Let $B_n \subset \Gamma$ be an increasing sequence of finite sets such that $\Gamma = \cup_n B_n$. For each $n \geq 1$ we may find an index $i(n) \in I$ such that $\| x_k^{i(n)} - x_k^{i(n)} \cdot g \| < 2^{-n}$ for all $g \in B_n$, and $k = 1, \ldots, d$. 

By Lemma~\ref{quasicentral} below there exists an increasing sequence $\alpha_n \in I_0(X)$ so that $0 \leq \alpha_n \leq 1$, $\alpha_n \to 1$ pointwise, and such that for all $g \in B_n$ and $k = 1, \ldots, d$ we have
\begin{equation}\label{eq:1}
\| (1 - \alpha_n) (y_k^{i(n)} - y_k^{i(n)} \cdot g) \| < 2^{-n}; \ \ \ \ \| \alpha_n - \alpha_n \cdot g \| < 2^{-n};
\end{equation}
\begin{equation}\label{eq:2}
\| (1 - \alpha_n) ( b^{i(n)} - \sum_{k = 1}^d (y_k^{i(n)} - g_k \cdot y_k^{i(n)} ) ) \| < 2^{-n}; \ \ \ \ \
\| \alpha_n - g_k \cdot \alpha_n  \| < 2^{-n}.
\end{equation}

For $k = 1, \ldots, d$ we define $y_k := \sum_{n \geq 1} (\alpha_{n + 1} - \alpha_n) y_k^{i(n)}$, and we define $b := \alpha_1 + \sum_{n \geq 1} (\alpha_{n + 1} - \alpha_n) b^{i(n)}$. Note that $\| 1 - b \| \leq 1/2$, since $\alpha_1 + \sum_{n \geq 1} (\alpha_{n + 1} - \alpha_n) = 1$.

If $g \in B_m$, then from (\ref{eq:1}) we have 
\begin{align}
& \| ( \sum_{n \geq m} (\alpha_{n + 1} - \alpha_n) y_k^{i(n)} ) - ( \sum_{n \geq m} (\alpha_{n + 1} - \alpha_n) y_k^{i(n)} ) \cdot g  \| \nonumber  \\
& \leq \sum_{n \geq m} \|  (\alpha_{n + 1} - \alpha_n) -  (\alpha_{n + 1} - \alpha_n) \cdot g \| \| x_k \|
+ \sum_{n \geq m} \| (\alpha_{n + 1} - \alpha_n) ( y_k^{i(n)} - y_k^{i(n)} \cdot g ) \| \nonumber \\
& \leq2 \sum_{n \geq m} 2^{-n}  \| x_k \|  + \sum_{n \geq m} \| (1 - \alpha_n) ( y_k^{i(n)} - y_k^{i(n)} \cdot g ) \| \nonumber \\ 
&\leq 2^{-m + 2} ( \| x_k \| + 1). \nonumber 
\end{align}
Thus, $y_k - y_k \cdot g \in I_0(X)$ for all $g \in \Gamma$, $k = 1, \ldots, d$.

Similarly, from (\ref{eq:2}) we have
\begin{align}
& \|  \sum_{n \geq m}  (\alpha_{n + 1} - \alpha_n) b^{i(n)} - \sum_{k = 1}^d ( \sum_{n \geq m}   (\alpha_{n + 1} - \alpha_n) y_k^{i(n)} ) - g_k \cdot ( \sum_{n \geq m}   (\alpha_{n + 1} - \alpha_n) y_k^{i(n)} )  \| \nonumber \\
 & \leq \sum_{k = 1}^d \sum_{n \geq m} \|  (\alpha_{n + 1} - \alpha_n)  - g_k \cdot (\alpha_{n + 1} - \alpha_n) \| \| x_k \| \nonumber \\
&+  \sum_{n \geq m} \| (\alpha_{n + 1} - \alpha_n) ( b^{i(n)} - \sum_{k = 1}^d (y_k^{i(n)} - g_k \cdot y_k^{i(n)} ) ) \| \nonumber \\
&\leq 2 \sum_{n \geq m} 2^{-n} \sum_{k = 1}^d \| x_k \| + \sum_{n \geq m} \| (1 - \alpha_n) ( b^{i(n)} - \sum_{k = 1}^d (y_k^{i(n)} - g_k \cdot y_k^{i(n)} ) ) \|  \nonumber \\
&\leq 2^{-m + 2} ( 1 + \sum_{k = 1}^d \| x_k \|). \nonumber 
\end{align}
Thus, $b - \sum_{k = 1}^d y_k - g_k \cdot y_k \in I_0(X)$. 

Hence, for each k, if we denote by $z_k \in \ell^\infty(\Gamma)/ I_0(X)$ the projection of $y_k \in \ell^\infty(\Gamma)$, then each $z_k$ is right $\Gamma$-invariant and $\| 1 - \sum_{k = 1}^d (z_k - g_k \cdot z_k) \| \leq 1/2$. This then rules out the existence of a left-$\Gamma$-invariant state on $C(X)^{\Gamma_r}$.

$(iii) \Rightarrow (ii)$. This follows from Lemma \ref{all proximal}.

$(ii) \Rightarrow (i)$. This is trivial.
\end{proof}

In the previous proof, we used the following lemmas:

\begin{lem}\label{Banach}
If $X$ is a Banach space and $(x_i)$ is a net of elements in $X$ which converges *-weakly to an element $x \in X^{**}$, then we have $\inf\{ \Vert y \Vert \mid y \in \conv(\{x_i\}) \} \leq \Vert x \Vert$.
\end{lem}
\begin{proof}
If $x = 0$, then $x \in X$ and the result is known. So let us assume that $x \neq 0$.
Assume that there is $\eps > 0$ such that $\inf\{\Vert y \Vert, y \in \conv(\{x_i\})\} \geq \Vert x \Vert(1 + \eps)$.
Then by Hahn-Banach separation theorem, there exists a non-zero linear functional $\varphi \in X^*$ that separates the open ball $B_X(0,\Vert x \Vert(1 + \eps))$ from $\conv(\{x_i\})$ in the following sense:
\begin{equation}\label{truc}\sup\{\varphi(a) \mid \Vert a \Vert < \Vert x \Vert(1 + \eps))\} \leq \inf\{\varphi(y) , y \in \conv(\{x_i\})\}.\end{equation}
Note that the left term in the above equation is nothing but $\Vert x \Vert(1+\eps)\Vert \varphi \Vert$, which is non-zero since $x$ and $\varphi$ are non-zero.
By Goldstine's theorem, we know that $x$ belongs to the weak* closure of $\{a \in X \mid \Vert a \Vert \leq \Vert x \Vert\}$. So $(1 + \eps)x$ belongs to the weak-* closure of $B_X(0,\Vert x \Vert(1 + \eps))$. In particular
\[(1+\eps)\varphi(x) \leq \inf\{\varphi(y) , y \in \conv(\{x_i\})\} \leq \varphi(x).\]
This implies that $\varphi(x) = 0$ and hence both terms in \eqref{truc} are equal to $0$, a contradiction.
\end{proof}

\begin{lem}\label{quasicentral}
Let $A$ be a unital $C^*$-algebra with a closed ideal $I \subset A$. Suppose $\Gamma$ is a countable group which acts on $A$ by $*$-automorphisms which preserve $I$. If $I_0 \subset I$ is a countable set, then there exists an increasing sequence $\alpha_n \in I$, with $0 \leq \alpha_n \leq 1$, such that $\| (1 - \alpha_n) a \| \to 0$ for all $a \in I_0$, and such that $\| \alpha_n - g \cdot \alpha_n \| \to 0$ for all $g \in \Gamma$. 
\end{lem}
\begin{proof}
This is essentially contained in the proof of Theorem 1 in \cite{A77} (see also Theorem I.9.16 in \cite{D96}). Fix an approximate unit $\{ \upsilon_\lambda \}_{\lambda \in \Lambda}$ for $I$, and note that this is also an approximate unit for $J = I \rtimes \Gamma$. The proof of Theorem 1 in \cite{A77} then shows that after passing to convex combinations we may obtain an approximate unit $\{ \alpha_\lambda \}_{\lambda \in \Lambda}$ of $I$ which is quasi-central in $A \rtimes \Gamma$. In particular it follows that $\| \alpha_\lambda - g \cdot \alpha_\lambda \| \to 0$ for all $g \in \Gamma$. 

Since $\Gamma$ and $I_0$ are countable we may then take a subsequence $\{ \alpha_n \}_{n \in \mathbb N}$ such that $\| (1 - \alpha_n) a \| \to 0$ for all $a \in I_0$, and such that $\| \alpha_n - g \cdot \alpha_n \| \to 0$ for all $g \in \Gamma$.
\end{proof}

%%%%%
\subsection{First examples and properties}

\begin{examp} \label{convergence}
Recall that a group $\Gamma$ is called a \emph{convergence group} if there exists a continuous action $\Gamma \actson K$ on a compact space $K$ having at least three points such that the action induced on the (locally compact) space of distinct triples of $K$ is proper. $\Gamma$ is said to be \emph{non-elementary} if it is infinite and the action on $K$ does not preserve (set-wise) a set with at most two elements.
Non-elementary convergence groups are properly proximal. 
\end{examp}
\begin{proof}
Take a continuous action $\Gamma \actson K$ on a compact space $K$ such that the action induced on the set of distinct triples of $K$ is proper. By \cite[Proposition 1.1]{Bo96} it follows that for any diffuse measure $\eta$ on $K$, we have $\partial\Gamma = \partial_\eta\Gamma$. So we need to check that there indeed exists a diffuse measure on $K$ and that there is no $\Gamma$-invariant measure on $K$. The latter fact follows easily from the fact that the action is non-elementary. To prove the existence of a diffuse measure, first note that the exists a Hausdorff compact perfect set inside $K$. Indeed, as mentioned in \cite[Section 2]{Bo96}, if $\Gamma$ is non-elementary then its limit set is perfect. So we can apply Lemma \ref{diffuse measure}.
\end{proof}

\begin{prop}
If $\Gamma$ admits a proper two-sided array into a non-amenable representation then it is properly proximal. 
\end{prop}
\begin{proof}
Denote by $\pi$ a unitary representation of $\Gamma$ and assume that there exists a proper two-sided array $b$ into $\pi$. Since $b$ is proper, the corresponding boundary piece, as defined in Lemma \ref{array piece}, is equal to $\partial \Gamma$. Then from Lemma \ref{array piece} there exists a unital completely positive map $\theta : B(H) \to C(\partial\Gamma)^{\Gamma_r}$ which is equivariant, in the sense that $\theta \circ \Ad(\pi(g)) = \lambda_g \circ \theta$ for all $g \in \Gamma$. Assuming that $\varphi$ is a left $\Gamma$-invariant state on $C(\partial\Gamma)^{\Gamma_r}$, we obtain that $\varphi \circ \theta$ is an $\Ad(\pi(\Gamma))$-invariant state on $B(H)$, showing that $\pi$ is amenable. So if $\pi$ is non-amenable, condition $(iii)$ of Theorem \ref{carac pp} is satisfied, showing that $\Gamma$ is properly proximal.
\end{proof}

\begin{examp}
A group with a proper cocycle into a non-amenable representation is properly proximal.
\end{examp}

\begin{examp}\label{bi-exact example}
Recall from \cite[Section 15]{BO08} that a group $\Gamma$ is said to be \emph{bi-exact} if it is exact and admits a proper two-sided array into its left regular representation, see also \cite[Proposition 2.7]{PV12b}. Thus we see that non-amenable bi-exact groups are properly proximal.
Alternatively, bi-exactness is characterized by the property that the left $\Gamma$-action on $C(\partial\Gamma)^{\Gamma_r}$ is topologically amenable. So there is no $\Gamma$-invariant state as soon as $\Gamma$ is non-amenable. 
\end{examp}

The above examples show that the following classes of groups are properly proximal:
\begin{itemize}
\item Non-elementary hyperbolic groups (being both convergence groups and bi-exact);
\item Arbitrary free products (being convergence groups);
\item The wreath product $\Z^{(\Gamma)} \rtimes \Gamma$,  for any bi-exact group $\Gamma$ (being again bi-exact by \cite{Oz04}, although it is not a convergence group). 
\end{itemize}
All these classes of groups admit some hyperbolicity properties. Even though Theorem \ref{carac  pp} shows that only one action is needed to define properly proximal groups, the flexibility of allowing several distinct actions significantly increases the class of examples for which we can prove this property. For instance we will prove in the next section that all lattices in all real semi-simple Lie groups with trivial center (e.g.\ $\SL_n(\Z)$ for all $n \geq 2$) are properly proximal groups.

\begin{prop}\label{stability}
We have the following stability properties.
\begin{enumerate}
\item A direct product of finitely many properly proximal groups is again properly proximal;
\item A co-amenable subgroup (e.g.\ a finite index subgroup) of a properly proximal group is properly proximal.
\item The class of properly proximal groups  is stable under commensurability up to finite kernels.
\end{enumerate}
\end{prop}
\begin{proof}
$(1)$. Consider two properly proximal groups $\Gamma_1$ and $\Gamma_2$ and denote by $\Gamma = \Gamma_1 \times \Gamma_2$. Extend the quotient maps $\Gamma \to \Gamma_i$ to a continuous maps $\pi_i: \Delta \Gamma \to \Delta \Gamma_i$, $i = 1,2$, on the Stone-\v{C}ech compactifications. Then we have $\partial \Gamma = \pi_1^{-1}(\partial \Gamma_1) \cup \pi_2^{-1}(\partial\Gamma_2)$ (going to infinity inside $\Gamma$ amounts to having at least one coordinate going to infinity).

Of course any action $\sigma$ of $\Gamma_i$ on a compact space $K$ gives rise to an action $\sigma \circ \pi_i$ of $\Gamma$ on $K$. Note that $\sigma \circ \pi_i$ admits a $\Gamma$-invariant measure if and only if $\sigma$ has a $\Gamma_i$-invariant measure. Moreover, for any measure $\eta \in \Prob(K)$ we have $\partial_\eta\Gamma = \pi_i^{-1}(\partial_\eta \Gamma_i)$, where the left-hand expression refers to the $\Gamma$-action $\pi_i \circ \sigma$ while the right-hand side refers to the $\Gamma_i$-action $\sigma$. So the statement holds true.

$(2)$. Assume that $\Gamma$ is properly proximal and take a co-amenable subgroup $\Lambda < \Gamma$. By Theorem \ref{carac pp}, there is no left-$\Gamma$-invariant state on $C(\partial\Gamma)^{\Gamma_r}$. The restriction map $C(\partial\Gamma) \to C(\partial\Lambda)$ is left and right $\Lambda$-equivariant, so the existence of a left $\Lambda$-invariant state on $C(Y_i)^{\Lambda_r}$ implies the existence of a left $\Lambda$-invariant state on $C(X_i)^{\Gamma_r}$. By co-amenability of the inclusion $\Lambda<\Gamma$, this would further imply the existence of a $\Gamma$-invariant state on $C(X_i)^{\Gamma_r}$, which is not the case. It follows that $\Lambda$ is properly proximal.

$(3)$. It follows from $(2)$ that a finite index subgroup of a properly proximal group is properly proximal. Conversely assume that $\Lambda < \Gamma$ is a finite index inclusion of groups with $\Lambda$ properly proximal. Replacing $\Lambda$ with a finite index subgroup if necessary, we may assume that it is normal inside $\Gamma$. 

First note that the inclusion $C(\partial\Gamma)^{\Gamma_r} \subset C(\partial\Gamma)^{\Lambda_r}$ admits a left-$\Gamma$-equivariant conditional expectation $E$. Indeed, denote by $F := \Gamma/\Lambda$ and observe that the right $\Gamma$-action $\rho$ on $C(\partial\Gamma)^{\Lambda_r}$ factorizes to an action of the finite group $F$, and that $C(\partial\Gamma)^{\Gamma_r}$ is exactly the subalgebra of fixed points for this action of $F$. So the averaging map $x \mapsto \sum_{g \in F} \rho_g(x)$ is a left-$\Gamma$-equivariant conditional expectation.

Choose a set $\widetilde{F} \subset \Gamma$ of representatives of the cosets of $\Lambda$ inside $\Gamma$. The equality $\ell^\infty(\Gamma) = \bigoplus_{g \in \widetilde{F}} \rho_g\ell^\infty(\Lambda)$ implies $C(\partial\Gamma) = \bigoplus_{g \in \widetilde{F}} \rho_gC(\partial\Lambda)$, and further $C(\partial\Gamma)^{\Lambda_r} = \bigoplus_{g \in \widetilde{F}} \rho_gC(\partial\Lambda)^{\Lambda_r}$, because each $g$ normalizes $\Lambda$. The ucp map $\phi: f \in C(\partial\Lambda)^{\Lambda_r} \mapsto \bigoplus_g \rho_g(f) \in C(\partial\Gamma)^{\Lambda_r}$ is left-$\Lambda$-equivariant.

To conclude we use that characterization from Theorem \ref{carac pp}.(iii). If there is a left-$\Gamma$-invariant state $\varphi$ on $C(\partial\Gamma)^{\Gamma_r}$, then $\varphi \circ E \circ \phi$ is left-$\Lambda$-invariant on $C(\partial\Lambda)^{\Lambda_r}$ contradicting the fact that $\Lambda$ is properly proximal. We thus deduce that proper proximality is stable under commensurability.

Take two discrete groups $\Gamma$ and $\Gamma'$ with a surjective homomorphism $\pi: \Gamma \to \Gamma'$ with finite kernel. Since $\pi$ has finite kernel, we have $\partial\Gamma = \pi^{-1}(\partial \Gamma')$, where we also denote by $\pi$ the continuous extension $\pi: \Delta\Gamma \to \Delta \Gamma'$. The restriction $\pi: \partial\Gamma \to \partial\Gamma'$ then gives an embedding $C(\partial\Gamma') \subset C(\partial\Gamma)$. In fact, we have the equality $C(\partial\Gamma') = C(\partial\Gamma)^F$, where $F = \ker(\pi)$. Moreover, the left and right $\Gamma$-actions leave the subalgebra $C(\partial\Gamma')$ globally invariant, on which these actions factor to $\Gamma'$-actions. In particular, we have $C(\partial\Gamma')^{\Gamma'_r} = C(\partial\Gamma)^{\Gamma_r}$, where the left-$\Gamma$-action restricted to $C(\partial\Gamma)^{\Gamma_r}$ factors to the left-$\Gamma'$-action on $C(\partial\Gamma')^{\Gamma'_r}$.
So there is a left-$\Gamma$-invariant state if and only if there is a left-$\Gamma'$-invariant state on $C(\partial\Gamma')^{\Gamma'_r}$. So we conclude from Theorem \ref{carac pp} that $\Gamma$ is properly proximal if and only if $\Gamma'$ is.
\end{proof}

\begin{prop}\label{inner amenable}
Properly proximal groups are not inner amenable.
\end{prop}
\begin{proof}
Assume that $\Gamma$ is an inner amenable amenable group. Then there exists a state $m: \ell^\infty(\Gamma) \to \C$ which is invariant under the conjugation action of $\Gamma$ and which vanishes on $c_0(\Gamma)$. So $m$ factorizes to a conjugation-invariant state on $C(\partial\Gamma) = \ell^\infty(\Gamma)/c_0(\Gamma)$. Its restriction to $C(\partial\Gamma)^{\Gamma_r}$ is then left-$\Gamma$-invariant. Applying Theorem \ref{carac pp}, we conclude that $\Gamma$ is not properly proximal.
\end{proof}

Proposition \ref{inner amenable} implies that the direct product of an infinite amenable group with an arbitrary group is never properly proximal. It also follows that an infinite direct sum of non-trivial groups is never properly proximal. Hence, proper proximality is not invariant under inductive limits.

%%%%%
\subsection{Linear properly proximal groups}
\label{linear}

In this section we study properly proximal linear groups.
Even if we are only interested in results for discrete countable groups, we will need to employ general Lie groups and algebraic groups. So before turning to the concrete examples of linear groups let us first mention a few facts about boundary pieces in locally compact groups.

Given a locally compact group $G$, we define its Stone-\v{C}ech compactification $\Delta G$ to be the Gelfand spectrum of the algebra $C_b(G)$ of bounded continuous functions on $G$ and its boundary $\partial G$ is then the spectrum of the quotient $C_b(G)/C_0(G)$, where $C_0(G)$ is the ideal of continuous functions going to $0$ at infinity.

Given a continuous action $G \actson K$ on a compact space and a measure $\eta \in \Prob(K)$, we may define $\partial_\eta G$ in a similar way as for the discrete case. However, in order to avoid the use of ultrafilters in this topological setting, we proceed as follows instead: since the action $G \actson K$ is continuous, the function $f_\eta: g \mapsto \int_K f(gx)d\eta$ is in $C_b(G)$ for all $f \in C(K)$. Then define
\[\partial_\eta G := \{ \omega \in \Delta G \mid \omega(f_\eta) = (\omega h)(f_\eta), \text{ for all } h \in G, f \in C(K)\}.\]

\begin{lem}\label{LC boundary}
Take a subgroup $\Gamma < G$ and denote by $\pi: \Delta \Gamma \to \Delta G$ the continuous map extending the embedding. Assume that $G$ acts continuously on a compact space $K$ and consider also the restricted $\Gamma$ action on $K$. Then for any $\eta \in \Prob(K)$, we have $\partial_\eta \Gamma \supset \pi^{-1}(\partial_\eta G)$. In the case where $\Gamma$ is discrete inside $G$, then $\pi$ is an embedding and the formula can be read as $\partial_\eta \Gamma \supset \partial_\eta G \cap \Delta \Gamma$.
\end{lem}
\begin{proof}
Let us prove that if $\Gamma$ is discrete inside $G$ then $\pi$ is an embedding. Note that it is enough to prove that the restriction map $\pi_*: C_b(G) \to \ell^\infty(\Gamma)$ is surjective.  Since $\Gamma$ is discrete we may find a neighborhood $U$ of the trivial element $e$ such that for all distinct elements $\gamma, \gamma' \in \Gamma$, we have $\gamma U \cap \gamma'U = \emptyset$. Take an arbitrary function $F \in C_b(G)$ such that $F(e) = 1$ and $F$ is supported on $U$.
Then if $f \in \ell^\infty(\Gamma)$ one verifies that the function $\tilde f \in C_b(G)$ defined below satisfies $\pi_*(\tilde f) = f$:
\[\tilde f(g) = \begin{cases*}  f(\gamma)F(\gamma^{-1}g) & if $g \in \gamma U$, $\gamma \in \Gamma$, \\  0 & otherwise.\end{cases*}\]  
This proves that $\pi_*$ is surjective. All the other assertions of the lemma are easy.
\end{proof}

Let us now prove the main lemma about proximality in linear groups over local fields. The rest of this section uses heavily the language of algebraic groups. We refer to Section \ref{section SL} below for a direct argument in the special case of $\SL_d(\R)$.

\begin{lem}\label{simple case}
Consider an almost $k$-simple, connected, simply connected algebraic group $\G$ over a local field $k$. Then there are finitely many proper parabolic $k$-subgroups $\bP_i < \G$ and measures $\eta_i \in \Prob(\G(k)/\bP_i(k))$ such that $\partial G = \bigcup_i \partial_{\eta_i} G$.
\end{lem}
\begin{proof}
We use Cartan decomposition of simple groups over local fields as presented in \cite[Theorem I.(2.2.1)]{Ma91}

Denote by $d$ the $k$-rank of $\G$. If $d = 0$, then $\G(k)$ is compact and there is nothing to prove. Assume that $d \geq 1$ and denote by $\bS \subset \G$ a maximal $k$-split torus, and denote by $\Phi$ the corresponding set of roots and by $\Phi^+$ the set of positive roots with respect to some order. Denote by $\Phi_0 = \{\alpha_1,\dots,\alpha_d\} \subset \Phi^+$ the set of simple roots.

Set $\hat{k} := \{x \in \R \mid x \geq 1\}$ if $k \simeq \R$ or $\C$, and $\hat{k} = \{ \beta^{-n} \mid n \in \N \}$ if $k$ is non-archimedean (here $\beta$ is a uniformizer of $k$). Denote by $S^+ := \{ s \in \bS(k) \mid \chi(s) \in \hat{k} \text{ for all } \chi \in \Phi^+\}$. Then according to \cite[Theorem I.(2.2.1)]{Ma91}, there exists a compact subset $M \subset \G(k)$ such that $\G(k) = M S^+ M$.

For each $i$, denote by $\theta_i  := \Phi_0 \setminus \{\alpha_i\}$, and by $\bP_i = \bP_{\theta_i}$ the corresponding parabolic subgroup, see \cite[Section I.1.2]{Ma91} and by by $\V^-_i$ the unipotent radical of the opposite parabolic subgroup of $\bP_i$. For simplicity we set $G := \G(k)$, $S = \bS(k)$, $P_i := \bP_i(k)$, $V^-_i := \V^-_i(k)$.

The projection map $p_i: G \to G/P_i$ is continuous, and $G$-equivariant. We consider its restriction to $V^-_i$. 
Since $S \subset P_i$, we have
\begin{equation}\label{conjugacy} p_i(svs^{-1}) = p_i(sv) = s \cdot p_i(v), \text{ for all } s \in S, v \in V^-_i.\end{equation}
It follows from \cite[Lemma IV.2.2]{Ma91} that the restriction of $p_i$ to $V_i^-$ is an open map from $V_i^-$ into $G/P_i$. Moreover, if we take any probability measure $\nu_i \in \Prob(V^-_i)$ equivalent to the Haar measure, its push-forward $\eta_i \in \Prob(G_i/P_i)$ under $p_i$ is a quasi-invariant measure for the $G$-action.

The lemma will follow from the Cartan decomposition and the next three claims.
The first claim is in the spirit of \cite[Lemma II.3.1]{Ma91}.

{\bf Claim 1.} For any compact set $C \subset V_i^-$ and any neighbourhood $O$ of the neutral element $1_{V_i^-}$, there exists a constant $A > 0$ such that every $s \in S^+$ for which $\vert \alpha_i(s) \vert_{k} \geq A$, we have $sCs^{-1} \subset O$.

Since we may apply the exponential map, see \cite[Proposition I.1.3.3]{Ma91}, it suffices to prove the analogous statement for the Lie algebra $\fL$ consisting of the $k$-points of the Lie algebra of $\V_i^-$. See also \cite[Proposition I.2.1.1]{Ma91}. 

Note that $\fL$ is spanned by the eigenvectors in the adjoint representation corresponding to the negative roots in $\Phi$ which admit a non-trivial coefficient at $\alpha_i$. In other words, it is spanned by (finitely many) vectors $v_\alpha$, such that $\Ad(s)(v_\alpha) = \alpha(s)v_\alpha$ for all $s \in \bS(k)$, where $\alpha$ ranges over roots $\alpha = \alpha_1^{n_1} \cdots \alpha_d^{n_d}$ for some integers $n_1,\dots, n_d \leq 0$ and $n_i < 0$. 

Take a norm $\Vert \cdot \Vert_{\fL}$ on the vector space $\fL$ compatible a the absolute value $\vert \cdot \vert_k$ on $k$. By the previous paragraph, we may find a constant $A_1 > 0$ such that 
\[\Vert \Ad(s)(v) \Vert_{\fL} \leq A_1 \vert \alpha_i(s)^{-1} \vert_k \Vert v \Vert_{\fL}, \text{ for all } v \in \fL, s \in S^+.\]
Take a compact subset $C \subset \fL$ and a neighbourhood $O \subset \cL$ of $0 \in \fL$. We may assume that $C = \{v \in \fL \mid \Vert v \Vert_{\fL} \leq A_2\}$ and $O = \{v \in \fL \mid \Vert v \Vert_{\fL} \leq a\}$ for some constants $A_2 > 0$, $a > 0$. Setting $A := A_1A_2/a$, we see that $\Ad(s)(K) \subset O$ for all $s \in S^+$ satisfying $\vert \alpha_i(s)\vert_k > A$. This concludes the proof of Claim 1.

{\bf Claim 2.} For any net $s_n \in S^+$ such that $\vert \alpha_i(s_n)\vert_k \to \infty$, and all convergent nets $(h_n)_n$ and $(k_n)_n$ in $G$, with respective limits $h$ and $k$, we have the following weak-* convergence 
\[\lim_n (h_ns_nk_n)\cdot \eta_i = \delta_{hP_i}.\]

To prove this claim, it suffices to show that for any open neighborhood $U \subset G/P_i$ of $hP_i$, we have $\lim_n \eta_i((h_ns_nk_n)^{-1}U) = 1$. Since the $G$-action on $G/P_i$ is continuous, for any such $U$, we may find a neighborhood $U_0$ of $P_i \in G/P_i$ such that for all $n$ large enough, we have $h_nU_0 \subset U$. So in fact, it is sufficient to prove that for every neighborhood $U_0 \subset G/P_i$ of $P_i$ and for every $\eps > 0$ there exists $n_0$ large enough so that $\eta_i(k_n^{-1}s_n^{-1}U_0) \geq 1 - \eps$ for all $n \geq n_0$.

Fix such $U_0$ and $\eps$, and take a relatively compact, open set $\Omega \subset V_i^-$ with such that $\nu_i(k^{-1}p_i(\Omega)) > 1-\eps$. Applying Claim 1, we find $n_1$ large enough so that for all $n \geq n_1$, we have $s_n \cdot p_i(\Omega) \subset U_0$. The set $C := p_i(\Omega)$ is an open set in $G/P_i$ such that $\eta_i(k^{-1}C) > 1 - \eps$ and $C \subset s_n^{-1} U_0$ for all $n \geq n_1$.

Since $\eta_i(k^{-1}C) > 1 - \eps$, $C$ is open and the measure $\eta_i$ is regular, a routine argument shows that there exists $n_2$ such that $\eta_i(k_n^{-1}C) > 1 - \eps$ for all $n \geq n_2$. Thus for all $n \geq \max(n_1,n_2)$ we have 
\[\eta_i(k_n^{-1}s_n^{-1}U_0) \geq \eta_i(k_n^{-1}C) \geq 1 - \eps,\]
as wanted.

{\bf Claim 3.} For any unbounded net $s_n \in S^+$, there exists $i$ such that $\vert \alpha_i(s_n) \vert_k \to \infty$.

This follows from the fact that the map $\alpha_1 \times \cdots \times \alpha_d: \bS \to (\GL_1)^d$ is a $k$-isomorphism, which restricts to a homeomorphism $\bS(k)$ onto $(k^*)^d$ and maps $S^+$ into the subset $\{(x_1,\dots,x_d) \mid \vert x_i \vert \geq 1 \text{ for all }i\}$.

To conclude the proof, let $\omega \in \Delta G$ and take a net $(g_n)_n$ in $G$ which converges to $\omega$. Then for each $n$ we may write the Cartan decomposition of $g_n$: $g_n = h_ns_nk_n$, with $h_n,k_n \in M$ and $s_n \in S^+$. Taking a subnet if necessary, we may assume that $h_n$ and $k_n$ converge to elements $h$ and $k$ respectively. By Claim 3, we may find some $i$ such that $\lim_n \vert \alpha_i(s_n)\vert_k = +\infty$. Now we may apply Claim 2, and we conclude
\[(\omega t)(f_{\eta_i}) = \lim_n \int f d(g_nt) \cdot \eta_i = f(hP_i) = \omega(f_{\eta_i}) \text{  for all } f \in C(G/P_i), t \in G.\]
This shows that $\omega \in \partial_{\eta_i}G$.
\end{proof}

\begin{prop}\label{discrete}
Consider finitely many local fields $k_i$, and semi-simple connected $k_i$-groups $\G_i$. Set $G_i := \G_i(k_i)$ for each $i$ and take a discrete subgroup $\Gamma$ in $G := \Pi_i G_i$ whose projection on each $G_i$ is Zariski dense. Then $\Gamma$ is properly proximal. In particular, lattices in semi-simple algebraic groups over local fields are properly proximal.
\end{prop}
\begin{proof} 
Let us start with several reductions to simpler cases.

{\sc Step 1.} We may assume that each $\G_i$ is simply connected. 

Indeed, for each $i$ denote by $\widetilde{\G}_i$ the simply connected cover of $\G_i$ and by $p_i: \widetilde{\G}_i \to \G_i$ the corresponding central isogeny (see \cite[Proposition I.1.4.11]{Ma91}). Then $p_i$ is defined over $k_i$, hence $p_i^{-1}(G_i) \subset \widetilde{\G}_i(\tilde{k_i})$ for some finite extension $\tilde{k}_i$ of $k_i$, see \cite[Corollary I.2.1.3]{Ma91}. Setting $\widetilde{G}_i := \widetilde{\G}_i(\tilde{k_i})$, we see that the map $p := \Pi_i p_i: \Pi_i \widetilde{G}_i \to \Pi_i G_i$ has finite kernel. Hence by Proposition \ref{stability} it suffices to check that $p^{-1}(\Gamma)$ is properly proximal. So replacing $\Gamma$ with $p^{-1}(\Gamma)$, $\G_i$ with $\widetilde{\G}_i$ and $k_i$ with $\tilde{k}_i$ we have reduced to the case where each $\G_i$ is simply connected.

{\sc Step 2.} We may assume that each $\G_i$ is almost $k_i$-simple.

This is a direct consequence of the fact that a simply connected semi-simple $k$-group decomposes into a direct product of almost $k$-simple $k$-groups, \cite[Proposition I.1.4.10]{Ma91}.

For each $i$ denote by $\Gamma_i$ the image of $\Gamma$ inside $G_i$ under the projection map $\pi_i : G \to G_i$.

{\sc Step 3.} We may assume that each $\Gamma_i$ has non-compact closure with respect to the locally compact topology on $G_i$.

 Denote by $F$ the set of indices $i$ for which $\overline{\Gamma_i}$ is non-compact in the locally compact topology of $G_i$. Then the projection map $\pi: G \to G' := \Pi_{i \in F} G_i$ satisfies:
\begin{itemize}
\item $\pi(\Gamma)$ is discrete in $G'$ and
\item $\ker(\pi) \cap \Gamma$ is finite.
\end{itemize}
Thanks to Proposition \ref{stability}, the second property above shows that $\Gamma$ is properly proximal if and only if $\pi(\Gamma)$ is. So, after replacing $\Gamma$ with $\pi(\Gamma)$, we may assume without loss of generality that each $\Gamma_i$ has non-compact closure with respect to the locally compact topology.

In this simplified setting, let us now define the family of actions which witness that $\Gamma$ is properly proximal. 
By Lemma \ref{simple case}, for each $i$, we may find parabolic subgroups $\bP_{i,j}$ of $\G_i$ and probability measures $\eta_{i,j}$ on the homogeneous spaces $K_{i,j} := G_i/\bP_{i,j}(k_i)$ such that $\partial G_i = \bigcup_j \partial_{\eta_{i,j}} G_i$ for all $i$. Consider the corresponding actions of $G$ on $K_{i,j}$ obtained by composing with the projection maps $\pi_i$. Extend continuously each $\pi_i$ to a map between the Stone-\v{C}ech compactifications $\pi_i: \Delta G \to \Delta G_i$. We have the equalities
\[\partial G = \bigcup_i \pi_i^{-1}(\partial G_i) = \bigcup_{i,j} \pi_i^{-1}(\partial_{\eta_{i,j}} G_i) = \bigcup_{i,j} \partial_{\eta_{i,j}} G.\]
Since $\Gamma$ is discrete inside $G$, $\partial \Gamma \subset \partial G$ and Lemma \ref{LC boundary} shows that $\partial \Gamma = \bigcup_{i,j} \partial_{\eta_{i,j}}\Gamma$.
Moreover, since $\Gamma_i$ is Zariski-dense in $G_i$ and has non-compact closure inside $G_i$, it follows from Furstenberg's Lemma \cite{Fu76}  that there is no $\Gamma$-invariant probability measure on $K_{i,j}$ for all $i,j$ (\cite[Corollary 3.2.19]{Zi84} for a proof of this precise statement). This concludes the proof of the first statement. For the second statement, note that if $\Gamma < \G(K)$ is a lattice in a semi-simple algebraic group $\G$ over a local field $k$, then up to taking a finite index subgroup we may assume that $\G$ is $k$-connected. Then the result follows from the first part.
\end{proof}

Thanks to \cite[Proposition 3.1.6]{Zi84}, the previous proposition implies that lattices in connected semi-simple real Lie groups with trivial center are properly proximal.

\begin{cor}\label{trivial radical}
A finitely generated subgroup $\Gamma < \GL_d(\bQ)$ with trivial solvable radical is properly proximal.
\end{cor}
\begin{proof}
Denote by $R \subset \bQ$ the ring generated by the entries of the elements of $\Gamma$ and by $k$ the subfield of $\bQ$ generated by $R$. Since $\Gamma$ is finitely generated, so is $R$ and $k$ is a number field. Denote by $G < \GL_d(\bQ)$ the Zariski closure of $\Gamma$. Thus $G = \G(\bQ)$ for some algebraic group $\G$ which is defined over $k$ (because $\Gamma$ is Zariski dense inside $G$).

Since $\Gamma$ has trivial solvable radical, it does not intersect the solvable radical of $G$. So moding out by the solvable radical of $G$ if necessary, we may as well assume that $G$ is semi-simple. Moreover, since the Zariski-connected component of the identity has finite index in $G$, we may assume that $G$ is Zariski-connected. So $\G$ is a connected semi-simple $k$-group.

Since $R$ is finitely generated, there are finitely many places $k_\nu$, $\nu \in \cS$ of $k$ such that the diagonal embedding of $R$ into $\Pi_\nu k_\nu$ is discrete. Since $\Gamma$ is contained in $\G(R)$, the diagonal embedding of $\Gamma$ into the product $\Pi_\nu \G(k_\nu)$ is discrete.

For all $\nu$, $\G(k_\nu)$ is semi-simple, so there exists a $k_\nu$-map with finite kernel from $\G(k_\nu)$ onto a product of almost simple $k_\nu$-groups. Again, since $\Gamma$ has trivial solvable radical, this map is injective on $\Gamma$. So we arrive at the situation of Proposition \ref{discrete}, giving the result.
\end{proof}

\begin{que} Let us ask now a few related questions.
\begin{enumerate}[(1)]
\item In the above Corollary, can the finite generation assumption on $\Gamma$ be removed? Is this the case if in the definition of proper proximality we allow countably many pieces $X_i$ (instead of only finitely many)? Note that this softer version of proper proximality would still be good enough for our applications. 
\item Can Corollary \ref{trivial radical} be extended to (finitely generated) subgroups of $\GL_d(\R)$ with trivial solvable radical? For example, is $\SL_d(\Z[t])$ properly proximal? Related to this question, we point out the reference \cite{GHW05} which proves a generalization of the discrete embedding $R \subset \Pi_\nu k_\nu$ used in the previous corollary.
\item As we saw before, properly proximal groups are non-inner amenable. In view of Tucker-Drob's results \cite[Theorem 13 and 14]{TD15}, one may wonder if the converse holds for linear groups: is any non-inner amenable (finitely generated) linear group properly proximal? In particular, does this hold for subgroups of $\GL_d(\bQ)$?
\end{enumerate}
\end{que}

%%%%%
\subsection{Von Neumann algebraic results on properly proximal groups}

\begin{thm}\label{main tech}
Assume that $\Gamma$ is a properly proximal group. Consider a trace preserving action $\sigma: \Gamma \actson (Q,\tau)$ on a tracial von Neumann algebra. Denote $M:=Q\rtimes\Gamma$

Then any weakly compact von Neumann subalgebra $P \subset M$ such that $\cN_M(P)\dpr$ contains $L\Gamma$ admits a corner that embeds into $Q$ inside $M$. %in the sense of Popa.
\end{thm}
\begin{proof}
%Denote by $M := Q \rtimes \Gamma$. 
Recall that $L^2(M) \simeq L^2(Q) \ot \ell^2\Gamma$ and that with this identification, $M$ is generated by $Q \ot 1$ and by the unitaries $u_g := \sigma_g \ot \lambda_g$, $g \in \Gamma$. Then $JMJ$ is generated by $JQJ$ and $1 \ot \rho(\Gamma)$, and $JQJ \subset B(L^2(Q)) \ovt \ell^\infty(\Gamma)$.

Assume that $P \subset M$ is a weakly compact inclusion such that $L\Gamma \subset \cN_M(P)\dpr$.  Then there exists a state $\varphi: B(L^2(M)) \to \C$ with the following properties:
\begin{enumerate}[(i)]
\item $\varphi$ is the canonical (normal) trace on $M$ and $JMJ$;
\item $\varphi(xT) = \varphi(Tx)$ for all $T \in B(L^2(M))$, $x \in P$;
\item $\varphi(uJuJT) = \varphi(TuJuJ)$ for all $T \in B(L^2(M))$ and all $u \in \cN_{M}(P)$.
\end{enumerate}

Assume that $P \nprec_M Q$. Then by Lemma \ref{intertwining}, (ii) above implies that $\varphi$ vanishes on $Q \ot \K(\ell^2\Gamma)$. 
Denote by $\pi:  \ell^\infty(\Gamma) \to \ell^\infty(\Gamma)/c_0(\Gamma) \simeq C(\partial\Gamma)$ the canonical projection, and by $B := \pi^{-1}(C(\partial\Gamma)^{\Gamma_r})$. 

Denote by $C := C^*(1 \ot B, \{u_g \mid g \in \Gamma\}) \subset B(L^2(M))$ and by $D := C^*(JQJ,1 \ot \rho(\Gamma)) \subset JMJ$. As explained before, $D$ is strongly dense inside $JMJ$. Note that $[x,y] \in Q \ot \K(\ell^2\Gamma)$ for all $x \in C$ and all $y \in D$.

{\bf Claim.} For all $x \in C$ and all $u \in \cN_M(P)$, we have $\varphi(ux) = \varphi(xu)$.

{\it Proof of the claim.} 
Fix such $x$ and $u$. Applying property (iii) above to the elements $u$ and $T = xJu^*J$ gives that $\varphi(uJuJxJu^*J) = \varphi(xu)$. Take a net $y_k \in D$ which converges to $Ju^*J$ in the strong topology of $JMJ$. Since $\varphi_{|JMJ}$ is the normal trace on $JMJ$, denoting $\|y\|_{\varphi}=\sqrt{\varphi(y^*y)}$ for $y\in {B(L^2(M))}$ and applying the Cauchy-Schwarz inequality, we get the following two computations:
\[\lim_k \vert \varphi(uJuJxJu^*J) - \varphi(uJuJxy_k)\vert \leq \lim_k \Vert (uJuJx)^* \Vert_{\varphi} \Vert Ju^*J - y_k \Vert_{\varphi} = 0.\]
\[\lim_k \vert \varphi(ux) - \varphi(uJuJy_kx)\vert = \lim_k \varphi(JuJ(Ju^*J - y_k)ux) \vert \leq \lim_k \Vert JuJ - y_k^*\Vert_{\varphi}\Vert ux\Vert_\varphi = 0.\]

As we observed, for all $k$, $xy_k - y_kx \in Q \ot \K(\ell^2\Gamma)$. Since $\varphi$ vanishes on $Q \ot \K(\ell^2\Gamma)$, we get that  $\vert \varphi(uJuJy_kx) -  \varphi(uJuJxy_k)\vert \leq \Vert xy_k - y_kx \Vert_\varphi = 0$. Therefore, we may combine the two computations above and get
\[\varphi(xu) = \varphi(uJuJxJu^*J) = \lim_k  \varphi(uJuJxy_k) = \lim_k  \varphi(uJuJy_kx) = \varphi(ux).\]
This proves the claim.

Since $\varphi$ is tracial on $M$, we obtain that  $\varphi(ux) = \varphi(xu)$, for all $x \in C$ and $u \in \cN_M(P)\dpr$. In particular, if $L\Gamma \subset \cN_M(P)\dpr$, we obtain $\varphi(u_gfu_g^*) = \varphi(f)$ for all $g \in \Gamma$ and $f \in 1 \ot \cB(\partial\Gamma)^{\Gamma_r}$. Since $u_gfu_g^* = \lambda_g(f)$ for all $g \in \Lambda$ and $f \in 1 \ot B$. We deduce that the restriction of $\varphi$ to $1 \ot B$ is a left $\Gamma$-invariant state, which vanishes on $1 \ot c_0(\Gamma)$. Hence it factors to a left invariant state on $C(\partial\Gamma)^{\Gamma_r}$, which contradicts the fact that $\Gamma$ is properly proximal, by Theorem \ref{carac pp}.
\end{proof}

Under the additional assumption that the properly proximal group $\Gamma$ is weakly amenable, we may use Popa and Vaes' result \cite[Theorem 5.1]{PV12a} to deduce Cartan rigidity results. Specifically, we obtain the following theorem.

\begin{thm}\label{C-rigid2}
Assume that $\Gamma$ is properly proximal and weakly amenable. Consider a trace preserving action $\Gamma \actson (Q,\tau)$ on a tracial von Neumann algebra. Denote $M:=Q\rtimes\Gamma$.

Then for any amenable von Neumann subalgebra $P \subset M$ such that $\cN_M(P)\dpr$ contains $L\Gamma$ we have that $P \prec_M Q$. In particular, $\Gamma$ is Cartan-rigid.
\end{thm}

In order to be able to apply \cite[Theorem 5.1]{PV12a}, we have to first establish the following statement about tensor products (as opposed to arbitrary crossed-products). 

\begin{prop}\label{PV prop}
Consider a group $\Gamma$ which is both properly proximal and weakly amenable. Consider also any tracial von Neumann algebra $B$ and set $M := B \ovt L\Gamma$.

Let $A \subset M$ be an amenable von Neumann subalgebra and assume that for any $g \in \Gamma$, there exists a unitary $w_g \in \cU(B)$ such that $w_g \ot u_g$ belongs to $\cN_M(A)\dpr$. %\footnote{This does not mean that $w_g \ovt u_g$ itself normalizes $A$.}. %Here we denoted by $u_g \in L\Gamma$, $g \in \Gamma$, the canonical group unitaries.

Then $A \prec_M B$.
\end{prop}

Before proving Proposition \ref{PV prop} let us deduce Theorem \ref{C-rigid2} from it.

\begin{proof}[Proof of Theorem \ref{C-rigid2}]
As in the statement of the theorem, set $M := Q \rtimes \Gamma$ and take $P \subset M$ an amenable subalgebra such that $\cN_M(P)\dpr$ contains $L\Gamma$.

Denote by $\Delta: M \to M \ovt L\Gamma$ the dual co-action of the action $\Gamma \actson Q$. Namely $\Delta$ is the *-homomorphism characterized by the formula
\[\Delta(au_g) := au_g \ot u_g, \quad \text{for all  } a \in Q, g \in \Gamma.\]
Set $B := M$ and $A := \Delta(P) \subset B \ovt L\Gamma$. By assumption the unitaries $\Delta(u_g) = u_g \ot u_g$ belong to the von Neumann algebra generated by the normalizer of $A$. So the proposition applies and shows that $A \prec_{B \ovt L\Gamma} B$. It is routine to check that this last fact implies that $P \prec_M Q$, proving the theorem.
\end{proof}

\begin{proof}[Proof of Proposition \ref{PV prop}]
Since $\Gamma$ is weakly amenable we may apply \cite[Theorem 5.1]{PV12a}. Let us introduce the corresponding notation. Recall that $M = B \ovt L\Gamma$ and $A \subset M$ is amenable. Denote $P := \cN_M(A)\dpr$, and recall that there are unitaries $w_g \in B$ such that $w_g \ot u_g \in P$ for all $g \in \Gamma$. Define $N$ as the von Neumann algebra acting on $L^2(M) \ot_A L^2(P)$ generated by $B$ and $P^{\op}$ and define $\cN := N \ovt L\Gamma$. Consider the two natural embeddings
\[\pi: a \ot u_g \in M \mapsto a \ot u_g \in \cN \quad \text{and} \quad \theta: y^{\op} \in P^{\op} \mapsto y^{op} \ovt 1 \in \cN.\]
It follows that $\cN$ is generated by the two commuting subalgebras $\pi(M)$ and $\theta(P^{\op})$. %Note that it would be equivalent to define $\cN$ as the von Neumann algebra generated by $M$ and $P^{\op}$ on $L^2(\langle M,e_B\rangle) \ot_A L^2(P)$.
With this notation, \cite[Theorem 5.1]{PV12a} tells us that there exists a net of normal states $\omega_i \in \cN_*$ such that 
\begin{itemize}
\item $\omega_i(\pi(x)) \to \tau(x)$ for all $x \in M$;
\item $\omega_i(\pi(a)\theta(\bar{a})) \to 1$ for all $a \in \cU(A)$ (here $\bar{a}$ denotes $(a^{\op})^*$);
\item $\Vert \omega_i \circ \Ad(\pi(u)\theta(\bar{u})) - \omega_i \Vert \to 0$ for all $u \in \cN_M(A)$.
\end{itemize}
Let us denote by $H$ a standard Hilbert space for $N$ with the corresponding anti-unitary involution $J: H \to H$. Then $\cN$ is standardly represented on $H \ot \ell^2(\Gamma)$, the associated anti-unitary involution $\cJ$ being defined by $\cJ(\xi \ot \delta_g) = J\xi \ot \delta_{g^{-1}}$, $\xi \in H$, $g \in \Gamma$.

For all $i$ denote by $\xi_i \in H \ot \ell^2(\Gamma)$ the canonical positive (unit) vector implementing the normal state $\omega_i$. Then as observed in \cite[Section 6]{PV12a}, the above properties of the states $\omega_i$ translate into properties of the vectors $\xi_i$ (see \cite[(6.1)-(6.3)]{PV12a}):
\begin{itemize}
\item $\langle\pi(x)\xi_i,\xi_i\rangle \to \tau(x)$ for all $x \in M$;
\item $\Vert\pi(a)\theta(\bar{a})\xi_i-\xi_i\| \to 0$ for all $a \in \cU(A)$;
\item $\Vert \pi(u)\theta(\bar{u})\cJ \pi(u)\theta(\bar{u})\cJ\xi_i - \xi_i \Vert \to 0$ for all $u \in \cN_M(A)$.
\end{itemize}

Further, define a state 
\[\Omega : T \in B(H \ot \ell^2(\Gamma)) \mapsto \lim_i \langle T\xi_i,\xi_i \rangle \in \C.\]
Note that by definition, $H$ is the standard space of $N = B \vee \theta(P^{\op})$. Then we see that the subalgebra $B \ovt B(\ell^2(\Gamma)) \subset B(H \ot \ell^2(\Gamma))$ commutes with $\theta(P^{op})$ and with $\cJ\theta(P^{\op})\cJ$. The state $\Omega$  is then easily seen to satisfy the following properties
\begin{enumerate}
\item $\Omega(\pi(x)) = \tau(x) = \Omega(\cJ\pi(x)\cJ)$ for all $x \in M$;
\item $\Omega(\pi(a)T) = \Omega(\theta(a^{\op})T) = \Omega(T\theta(a^{\op})) = \Omega(T\pi(a))$ for all $T \in B \ovt B(\ell^2(\Gamma))$, $a \in A$;
\item $\Omega(\text{Ad}(\pi(u)\cJ\pi(u)\cJ) (T)) = \Omega(\text{Ad}(\theta(u^{\op})\cJ \theta(u^{op})\cJ)(T)) = \Omega(T)$, for all $u \in \cN_M(A)$ and all $T \in B \ovt B(\ell^2(\Gamma))$.
\end{enumerate}
For the sake of a contradiction, assume that $A \nprec_M B$.

{\bf Claim 1.} $\Omega$ vanishes on $B \ot_{\min} \K(\ell^2(\Gamma))$.

Let us forget about the $\cJ$-map for now and only look at conditions $(1)$ and $(2)$ above. Then the algebra $B \ovt B(\ell^2(\Gamma))$ is isomorphic to the basic construction $\langle M,e_B\rangle$ and the embedding $\pi$ corresponds exactly to the canonical embedding $M \subset \langle M,e_B\rangle$. Then $\Omega$ becomes an $A$-central state on $\langle M,e_B\rangle$, which is tracial on $M$. The claim then follows from Lemma \ref{intertwining}.

Now we continue the proof in a similar fashion to the proof of Theorem \ref{main tech}. Denote by $\pi:  \ell^\infty(\Gamma) \to \ell^\infty(\Gamma)/c_0(\Gamma) \simeq C(\partial\Gamma)$ the canonical projection, and by $A := \pi^{-1}(C(\partial\Gamma)^{\Gamma_r})$.

%The restriction of $\Omega$ to $1 \ot \ell^\infty(\Gamma)$ defines a measure $\mu \in \Prob(\Delta\Gamma)$ which is supported on $\partial \Gamma$.
%By assumption $\Gamma$ is virtually properly proximal hence there is no left $\Gamma$-invariant state on $\cB(\partial\Gamma)^{\Gamma_r}$.
%Extend $\Omega$ to a normal state on $(B(H\otimes\ell^2(\Gamma)))^{**}$, still denoted by $\Omega$. Then $(B \ovt B(\ell^2(\Gamma)))^{**}\subset (B(H\otimes\ell^2(\Gamma)))^{**}$, and condition $(3)$ holds for every $T\in (B \ovt B(\ell^2(\Gamma)))^{**}$ and $u \in \cN_M(A)$.

We will check that the restriction of $\Omega$ to $1 \ot A$ is left $\Gamma$-invariant, contradicting proper proximality by Claim 1 and Theorem \ref{carac pp}.

Denote by $C$ the C*-subalgebra of $B \ovt B(\ell^2(\Gamma))$ generated by $\pi(M)$ and by $1 \ot A$. One checks that $\Vert [x,y]\Vert_\varphi = 0$ for all $x \in C$ and $y \in D:= \cJ\pi(B \ot_{\min} C^*_r(\Gamma))\cJ$. Then one proves the following claim exactly as in the proof of Theorem \ref{main tech}.

{\bf Claim 2.} For all $x \in C$ and all $u \in \cN_M(A)$ we have $\Omega(x\pi(u))=\Omega(\pi(u)x) $.

From this claim it follows that $\Omega(1 \ot \lambda_g(f)) = \Omega((w_g \ot u_g)(1 \ot f) (w_g \ot u_g)^*) = \Omega(1 \ot f)$ for all $g \in \Gamma$ and $f \in \cB(\partial\Gamma)^{\Gamma_r}$. This concludes the proof.
\end{proof}

Observe now that Theorem \ref{wc cartan} follows from Theorem \ref{main tech} and Theorem \ref{C-rigid} follows from Theorem \ref{C-rigid2}, thanks to the fact that two Cartan subalgebras $A,B$ in a given tracial von Neumann algebra $M$ are unitary conjugate if and only if $A \prec_M B$, \cite[Theorem A.1]{Po01}.

%%%%%%%%%%
\section{Bi-exactness towards a boundary piece}
Proper proximality for groups $\Gamma$ requires the existence of boundary pieces $X\subset\partial\Gamma$ which satisfy the following weak, global property: the left $\Gamma$-action on $C(X)^{\Gamma_r}$ does not admit an invariant state.  In this section, we investigate a stronger property for boundary pieces, called bi-exactness, obtained by assuming that the above action is amenable. This is motivated by a result from \cite[Section 15]{BO08} showing that $\Gamma$ is bi-exact precisely when the left $\Gamma$-action on $C(\partial\Gamma)^{\Gamma_r}$ is amenable. The results in this section are direct adaptations of \cite[Section 15]{BO08}. Nevertherless, we will recall most of the proofs for the convenience of the reader.

\begin{defn}\cite{AD87} \label{amenable}
A continuous action $\Gamma \actson K$ on a compact space is said to be \emph{amenable} if there exists a net of continuous maps $P_n: K \to \Prob(\Gamma)$ such that
\[\lim_n \sup_{x \in K} \Vert P_n(gx) - g \cdot P_n(x)\Vert_1 = 0, \text{ for all } g \in \Gamma.\]
\end{defn}

\begin{defn}
Given a group $\Gamma$ and a boundary piece $X \subset \partial \Gamma$, we say that $\Gamma$ is \emph{bi-exact towards $X$} if the left $\Gamma$-action on the Gelfand spectrum of $C(X)^{\Gamma_r}$ is amenable.
\end{defn}

\begin{rem}
In \cite[Section 15]{BO08} the notion of bi-exactness {\it relative} to a family of subgroups $\cG$ of $\Gamma$ is considered. One can check that this notion is equivalent to bi-exactness towards the boundary piece $X(\cG)$ given in Example \ref{relative piece}.
\end{rem}

Let us provide several equivalent formulations of this directional version of bi-exactness. This is reminiscent of \cite[Proposition 15.2.3]{BO08} and related facts. We provide a slightly different argument, which does not rely on Choi-Effros lifting theorem, nor on Voiculescu's theorem.

\begin{thm}\label{bi-exact equivalence}
Take a countable group $\Gamma$ with a non-empty boundary piece $X \subset \partial \Gamma$. The following assertions are equivalent.
\begin{enumerate}[(i)]
\item $\Gamma$ is bi-exact towards $X$;
\item There exists an amenable action $\Gamma \actson K$ and $\eta \in \Prob(K)$ such that $X \subset \partial_\eta \Gamma$;
\item $\Gamma$ is exact and there exists a map $\mu: \Gamma \to \Prob(\Gamma)$ such that 
\[\lim_{g \to \omega} \Vert \mu(sgt) - s \cdot \mu(g)\Vert_1 = 0, \text{ for all } s,t \in \Gamma, \omega \in X.\]
%\item $\Gamma$ is exact and there exists a two-sided array into the left regular representation $b : \Gamma \to \ell^2(\Gamma)$ whose boundary piece (see Lemma \ref{array piece}) contains $X$.
%\item The left-right action $\Gamma \times \Gamma \actson X$ is amenable.
\end{enumerate}
\end{thm}
\begin{proof}
$(i) \Rightarrow (ii)$. This follows from Lemma \ref{all proximal}. 

$(ii) \Rightarrow (iii)$. Assuming that condition $(ii)$ holds, we will first prove, for each finite subset $E \subset \Gamma$ and each $\eps > 0$, the existence of a map $\mu: \Gamma \to \Prob(\Gamma)$ satisfying the following conditions:

\begin{enumerate}[(a)]
\item $\sup_{g \in \Gamma} \Vert \mu(sg) - s \cdot \mu(g)\Vert_{1} < \eps$, for all $s \in E$;
\item For all $t \in E$, the set $\{ g \in \Gamma \, \vert \, \Vert\mu(gt) - \mu(g)\Vert_1 \geq \eps\}$ is small relative to $X := \partial_\eta\Gamma$, in the sense that its closure inside the Stone-\v{C}ech compactification $\Delta\Gamma$ does not intersect $X$.
\end{enumerate}

Fix a finite set $E \subset \Gamma$ and $\eps > 0$.
Consider maps $P_n: K \to \Prob(\Gamma)$ as in Definition \ref{amenable} and define for all $n$, $\mu_n: \Gamma \to \Prob(\Gamma)$ by the formula $\mu_n(g) := \int_K P_n(gx) d\eta$. By definition of the net $P_n$, $\mu_n$ satisfies (a) for $n$ large enough. Let us check that $\mu_n$ satisfies condition (b) for every $n$.

Fix $n$. Since $K$ is compact and $P_n$ is continuous, we may find a finite set $F$ such that 
\[\sup_{x \in K} \Vert P_n(x) - P_n(x)_{|F} \Vert_1 < \eps/3.\]
Here we wrote $P_n(x)_{|F}$ to denote the restriction of the measure $P_n(x)$ to the set $F$. Note that we may view any measure on $F$ as an element in a finite dimensional vector space, namely the dual of $\C^F$. 

Fix $t \in \Gamma$ and denote by $A_t := \{g \in \Gamma \, \vert \, \Vert\mu_n(gt) - \mu_n(g)\Vert_1 \geq \eps\}$. Assume by contradiction that $X \cap \overline{A_t} \subset \Delta\Gamma$ is non-empty. Then it contains an element $\omega$.
By the triangle inequality, we have for all $g \in A_t$,
\begin{equation}\label{At}\left\Vert \int_K P_n(gtx)_{|F}d\eta - \int_K P_n(gx)_{|F}d\eta \right\Vert_{1} \geq \eps/3.\end{equation}
Now, since the function $f: x \in K \mapsto P_n(x)_{|F} \in (\C^F)^*$ is a continuous function into a finite dimensional vector space, and since $\lim_{g \to \omega} (gt)\eta - g\eta = 0$ (weakly), we have the convergence $\lim_{g \to \omega} \Vert \int_K f d(gt\eta) - \int_K f d(g\eta)\Vert_1 = 0$. This clearly contradicts \eqref{At}

Having established the existence of the map $\mu$, we may now apply, {\it mutatis mutandis}, the procedure described in \cite[Exercise 15.1.1]{BO08} to deduce $(iii)$.

$(iii) \Rightarrow (i)$. Fix $\mu$ as in $(iii)$ and consider the map $\mu_*: \ell^\infty(\Gamma) \to \ell^\infty(\Gamma)$ defined by $\mu_*(f): g \mapsto \int_\Gamma fd\mu(g)$. From the properties of $\mu$ it follows that the composition $\psi := \phi \circ \mu_*$ of $\mu_*$ with the restriction map $\phi: \ell^\infty(\Gamma) \to C(X)$ is a positive map which satisfies $\psi(s \cdot f \cdot t) = s \cdot \psi(f)$. In particular, $\psi$ ranges into $C(X)^{\Gamma_r}$ and it is left equivariant. Since $\Gamma$ is exact, the action $\Gamma \actson \ell^\infty(\Gamma)$ is amenable, and so is the left action $\Gamma \actson C(X)^{\Gamma_r}$.
\end{proof}

We also mention the following generalization of \cite[Lemma 15.1.4]{BO08}, which relates bi-exactness to property AO. The proof is the same.

\begin{lem}\label{ucp lift}
Consider an exact group $\Gamma$ and a boundary piece $X \subset \partial \Gamma$. Then $\Gamma$ is bi-exact towards $X$ if and only if there exists a u.c.p. map 
\[\theta: C^*_\lambda(\Gamma) \ot_{\min} C^*_\rho(\Gamma) \to B(\ell^2\Gamma)\]
such that $\theta(x \ot y) - xy \in \K(\Gamma;X)$ (see Definition \ref{compact}) for all $x \in C^*_\lambda(\Gamma)$, $y \in C^*_\rho(\Gamma)$.
\end{lem}

As a corollary, we deduce the following solidity type result.

\begin{thm}\label{relatively solid}
Consider a group $\Gamma$ with a boundary piece $X \subset \partial\Gamma$. Assume that $\Gamma$ is bi-exact towards $X$.

For any net of unitaries $(u_n) \subset \cU(L\Gamma)$ with positive mass on $X$ (viewed as a net in $\ell^2\Gamma$), the relative commutant $(u_n)' \cap L\Gamma$ has an amenable direct summand.
\end{thm}
\begin{proof}
Assume that $\Gamma$ is bi-exact relative to $X$. Pick a net of unitaries $(u_n)_n \in L\Gamma$ with positive mass on $X$ and denote by $Q$ its relative commutant, $Q = (u_n)' \cap L\Gamma$. Define a state $\varphi: B(\ell^2\Gamma) \to \C$ by the formula $\varphi(T) = \lim_n \langle T \hat u_n, \hat u_n\rangle$, for all $T \in B(\ell^2\Gamma)$. Note that $\varphi(uJuJ) = 1$ for all $u \in \cU(Q)$, and $\varphi$ is tracial on $L\Gamma$ and $R\Gamma$.

Recall the notation $q_X = 1_X$ from \ref{nota1} and the notation $I(X)\subset A_\Gamma$ from subsection \ref{compactt}. We may extend $\varphi$ to a normal state on $B(H)^{**}$ so that $\varphi(q_X)$ makes sense. By assumption $\varphi(q_X) \neq 0$. Hence we may define a state $\psi: A_\Gamma \to \C$ by the formula $\psi(a) = \varphi(q_Xa)/\varphi(q_X)$ for all $a \in A_\Gamma$. 

By Lemma \ref{ucp lift}, since $\Gamma$ is bi-exact towards $X$, the map $x \ot y \in C^*_\lambda(\Gamma) \ot C^*_\rho(\Gamma) \mapsto q(xy) \in A_\Gamma / I(X)$ is min-continuous (we implicitly used that $I(X)=A_\Gamma\cap \K(\Gamma;X)$ by  Lemma \ref{au}). Here $q$ denotes the quotient map $A_\Gamma \to A_\Gamma / I(X)$. Since $\psi$ vanishes on $I(X)$, the following formula defines a continuous state:
\[\tpsi: x \ot y \in C^*_\lambda(\Gamma) \ot_\mini C^*_\rho(\Gamma) \mapsto \psi(xy) \in \C.\]
Note that $\tpsi$ is subtracial on $C^*_\lambda(\Gamma)$ and on $C^*_\rho(\Gamma)$.

Since $\Gamma$ is exact, $C^*_\lambda(\Gamma)$ has Property C from \cite[Section 9]{BO08}. Hence we have an embedding 
\[{C^*_\lambda(\Gamma)}^{**} \ot_{\mini} {C^*_\rho(\Gamma)}^{**} \subset { \left( C^*_\lambda(\Gamma) \ot_\mini C^*_\rho(\Gamma)\right)}^{**}.\]
In particular we get that the state $\tpsi$ extends to a state on $L\Gamma \ot_{\mini} R\Gamma$ which is normal on $L\Gamma$ and $R\Gamma$. Moreover, since the initial state $\varphi$ satisfies $\varphi(uJuJ) = 1$ for all $u \in \cU(Q)$, we find that $\tpsi(u \ot \bar{u}) = 1$ for all $u \in \cU(Q)$. Here $\bar{u}$ means ${u^*}^{\op}$, where we identified $R\Gamma$ with the opposite algebra of $L\Gamma$.

Representing $L\Gamma \ot_{\mini} R\Gamma$ standardly on $\ell^2\Gamma \ot \ell^2\Gamma$, we may approximate $\tpsi$ by vector states (by Glimm's lemma) : $\tpsi(x) = \lim_n \langle x \xi_n, \xi_n \rangle$, $x \in L\Gamma \ot_{\mini} R\Gamma$, for some unit vectors $\xi_n \in \ell^2(\Gamma) \ot \ell^2(\Gamma)$. These vectors satisfy $\lim_n \langle (u \ot \bar{u})\xi_n,\xi_n\rangle = 1$, which implies that $\lim_n \Vert (u \ot 1)\xi_n - (1 \ot u^{op}\xi_n\Vert = 0$. Then we see that the state $\Omega : B(\ell^2(\Gamma)) \to \C$ defined by $\Omega(T) = \lim_n \langle T \ot 1 \xi_n, \xi_n\rangle$ is normal on $L\Gamma$ and $Q$-central, proving that $Q$ has an amenable direct summand.
\end{proof}

%%%%%
\subsection{Patching bi-exactness}

Let us conclude this section with a discussion about patching arguments.
At a first glance one may be tempted to proceed as for proper proximality and define the class of ``patched bi-exact groups'' of groups whose Stone-\v
{C}ech boundary can be covered with finitely many pieces $X_i$ such that $\Gamma$ is bi-exact towards each $X_i$. We show here that such a notion is in fact equivalent to genuine bi-exactness. 

This relies on the following generalization of a result of Popa and Vaes, \cite[Proposition 2.7]{PV12b}. The proof is exactly the same.

\begin{prop}
Consider a discrete group $\Gamma$ and a boundary piece $X \subset \partial \Gamma$. The following are equivalent.
\begin{enumerate}[(i)]
\item There exists a map $\mu: \Gamma \to \Prob(\Gamma)$ as in Theorem \ref{bi-exact equivalence}.$(iii)$ ;
\item There exists a two-sided array into the regular representation $b : \Gamma \to \ell^2(\Gamma)$ which is proper towards $X$, meaning that the corresponding boundary piece introduced in Lemma \ref{array piece} contains $X$;
\item There exists a  unitary representation $\rho:\Gamma\rightarrow\mathcal U(K)$, weakly contained in the regular representation of $\Gamma$, and an array $q : \Gamma \to K$ which is proper towards $X$. 
\end{enumerate}
\end{prop}

\begin{prop}\label{bi ex patching}
Consider a discrete group $\Gamma$ and finitely many boundary pieces $X_i \subset \partial \Gamma$, $i = 1,\dots,n$. Assume that for each $i$, $\Gamma$ is bi-exact towards $X_i$. Then $\Gamma$ is bi-exact towards $\bigcup_i X_i$. 
\end{prop}
\begin{proof}
By assumption, we have for each $i$ an array $q_i: \Gamma \to \ell^2(\Gamma)$ which is proper towards $X_i$. Then the direct sum $q: \Gamma \to \oplus_i \ell^2(\Gamma)$ defined by $q(g) := \oplus_i q_i(g)$, $g \in \Gamma$, is an array which is clearly proper towards $\bigcup_i X_i$.
\end{proof}

This allows us to provide some new perspective on the following result of Ozawa, \cite{Oz09}.

\begin{cor}[Ozawa]
The group $\Gamma := \Z^2 \rtimes \SL_2(\Z)$ is bi-exact. 
\end{cor}
 \begin{proof}
Denote by $\Lambda := \SL_2(\Z)$ and by $\pi: \Delta\Gamma \to \Delta \Lambda$ the continuous extension of the projection map $\Gamma \to \Lambda$. We define 
\[X_1 := \pi^{-1}(\partial\Lambda) \text{ and } X_2 := \partial\Gamma \cap \pi^{-1}(\Lambda),\]
so that $\partial\Gamma = X_1 \cup X_2$.

First, note that the action $\sigma: \Lambda \actson \bP^1$ by homography is topologically amenable and it is a convergence action in the sense of Example \ref{convergence}. Hence the $\Gamma$-action obtained by composing $\sigma$ with the quotient map $\Gamma \to \Lambda$ is still amenable (because $\Z^2$ is amenable) and for any diffuse measure $\eta_1 \in \Prob(\bP^1)$, $X_1 = \partial_{\eta_1}\Gamma$. Hence $\Gamma$ is bi-exact towards $X_1$ by Theorem \ref{bi-exact equivalence}.

Second, embed $\Gamma$ inside $G := \SL_3(\R)$ in the usual way, and denote by $P <G$ the subgroup of upper triangular matrices, so that $G/P$ is the flag variety. Then $\Gamma \actson G/P$ is amenable and there exists $\eta \in \Prob(G/P)$ such that $X_2 \subset \partial_\eta\Gamma$, see the proof of Corollary \ref{SLd solid} below. Hence $\Gamma$ is bi-exact towards $\overline{X_2}$, by Theorem \ref{bi-exact equivalence}.

Proposition \ref{bi ex patching} then shows that $\Gamma$ is bi-exact.
\end{proof}

%%%%%%%%%%
\section{Application to the von Neumann algebras of $\SL_d(\Z)$}
\label{section SL}

In this section, we present a series of applications to the study of the von Neumann algebras of $\SL_d(\Z)$.
Let us start by specializing Lemma \ref{simple case} to the case of $\SL_d(Z)$.

From now on we fix $d \geq 3$ and set $\Gamma = \SL_d(\Z)$.

%%%%%
\subsection{Description of the canonical boundary pieces}

Denote by $G := \SL_d(\R)$ and let $K := \SO(d)$ be its maximal compact subgroup.

For all tuple $\bar k = (k_1,\dots, k_l)$ of integers $0 < k_1 < k_2 < \cdots < k_l < d$ denote by $P_{\bar{k}} < G$ the parabolic subgroup which stabilizes the subspaces $\mathbb R^{k_j} \times \{ 0 \}$, i.e., $P_{\bar{k}}$ consists of all matrices in $SL_d(\mathbb R)$ of the form 
\[
\begin{bmatrix}
GL_{k_1 - k_0}(\mathbb R) & * & * & \cdots & * \\
0 & GL_{k_2 - k_1}(\mathbb R) & * & \cdots & * \\
\vdots & \vdots & \vdots & \ddots & \vdots \\
0 & 0 & 0 & \cdots & GL_{k_l - k_{l - 1}} (\mathbb R)
\end{bmatrix}.
\]
Set $K_{\bar k} := G/P_{\bar k}$ and denote by $\eta_{\bar k} \in \Prob(K_{\bar k})$ the unique $K$-invariant probability measure on $K_{\bar k}$. The fact that it is $K$-invariant is not relevant to us, but it implies that $\eta_{\bar k}(Y) = 0$ for every proper algebraic subvariety of $K_{\bar k}$. In a sense this condition is similar to the diffuseness condition appearing in Example \ref{convergence}, and allows to avoid the unstable subvarieties that appear from the dynamics $\Gamma \actson K_{\bar k}$.

For each such tuple $\bar k$, denote by $X_{\bar k} := \partial_{\eta_{\bar k}}\Gamma$ the corresponding boundary piece. 

Given $g \in G$, recall that its singular values $s_i(g)$, $1 \leq i \leq n$ are the eigenvalues of $\sqrt{g^t g}$. We order them in such a way that $s_1(g) \geq s_2(g) \geq \dots \geq s_d(g) > 0$. These are easily seen to be the diagonal values of $G$ in the Cartan decomposition (also called $KAK$-decomposition).

The following lemma describes the boundary pieces $X_{\bar k}$ in terms of the ratios between singular values of group elements. It essentially goes back to Furstenberg's proof of Borel's density theorem \cite{Fu76}.

\begin{lem}
Fix a tuple $\bar k = (k_1,\dots, k_l)$. Then for any neighborhood $U \subset \Delta\Gamma$ of $X_{\bar k}$, there exists $C > 0$ such that
\[\{ g \in \Gamma \mid \frac{s_{k_j}(g)}{s_{k_{j+1}}(g)} \geq C \text{ for all } j \} \subset U.\]
\end{lem}
\begin{proof}
Fix a neighborhood $U$ of $X_{\bar k}$ inside the Stone-C\v{e}ch compactification of $\Gamma$. By definition of $X_{\bar k} = \partial_{\eta_{\bar k}}\Gamma$, we may choose a finite set $\cF \in C(K_{\bar k})$, a finite subset $E \subset \Gamma$ and $\eps > 0$ such that 
\[U_0 := \left\{ g \in \Gamma \, \mid \, \max_{f \in \cF, h \in E} \left\vert \int f d(g\eta_{\bar k}) - \int f d(gh\eta_{\bar k}) \right\vert \leq \eps\right\} \subset U.\]
If the lemma does not hold for this choice of $U$, then we may find a sequence $(g_n)$ in $\Gamma$ such that $g_n \notin U$ for all $n$ but $\lim_n s_{k_j}(g_n)/s_{k_{j+1}}(g_n) = +\infty$ for all $j$. For each $n$, we may write 
\[g_n = a_n\diag(s_1(g_n), \dots , s_d(g_n)) b_n, \text{ with } a_n ,b_n \in K.\]
Since $K$ is compact we may replace $g_n$ by a subsequence and assume that $(a_n)$ and $(b_n)$ converge to elements $a$ and $b$ respectively.

We let $F(k_1, k_2, \ldots, k_l)$ denote the flag variety consisting of flags with signature $(k_1, k_2, \ldots, k_l)$, and we let ${\rm Gr}(k_j, d)$ denote the Grassmannian of $k_j$-dimensional subspaces in $\mathbb R^d$.  We denote by $\pi_{k_j}$ the projection map from a flag of signature $(k_1, k_2, \ldots, k_l)$ onto its $k_j$-dimensional subspace. 

Since $\eta_{\bar k}$ is the $K$-invariant measure it follows that given a subspace $V \subset \mathbb R^d$ with $\dim V = d - k_j$, we have $\pi_{k_j} \eta_{\bar k} (\{ W \in {\rm Gr}(k_j, d) \mid W \cap V \not= \{ 0 \} \}) = 0$. On the other hand, if $W \in {\rm Gr}(k_j, d)$ is such that $W \cap ( \{ 0 \} \times \mathbb R^{d - k_j} ) = \{ 0 \}$ then we have that 
\[ a^{-1}g_n b^{-1}W = a^{-1} a_n \diag(s_1(g_n), \dots , s_d(g_n)) b_n b^{-1} W \to \mathbb R^{k_j} \times \{ 0 \}.\]
It therefore follows that for all $g \in G$ we have $ a^{-1}g_n b^{-1} g\pi_{k_j}\eta_{\bar k} \to \delta_{ \mathbb R^{k_j} \times \{ 0 \} }$. Hence, for all $g \in G$ we have that $g_n g \eta_{\bar k} \to \delta_{a F}$ where $F$ is the standard flag of signature $(k_1, k_2, \ldots, k_l)$, so that $g_n \in U_0$ for $n$ large enough, which contradicts our assumption.
\end{proof}

When the tuple $\bar{k}$ is a $1$-tuple $\bar k = (i)$, we just write $X_i$ instead of $X_{(i)}$.

\begin{cor}
For all $\omega \in \partial\Gamma$ and $1 \leq i \leq n$, if $\lim_{g \to \omega} s_i(g)/s_{i+1}(g) = +\infty$, then $\omega \in X_i$.
So the sets $X_i$ cover the Stone-C\v{e}ch boundary: $\partial\Gamma = \bigcup_{i = 1}^{d-1} X_{i}$.
\end{cor}
\begin{proof}
The first fact is an immediate consequence of the previous lemma. To deduce the second fact, we just need to observe that for every $\omega \in \partial\Gamma$, there exists an index $i$ such that $\lim_{g \to \omega} s_i(g)/s_{i+1}(g) = +\infty$. Indeed, since every element of $G$ has determinant $1$, for every $C > 0$, the set $\{ g \in G \mid s_i(g)/s_{i+1}(g) < C \text{ for all } i\}$ is bounded in $G$. Hence its intersection with $\Gamma$ is finite.
\end{proof}

In the special case of $\SL_3(\Z)$, we denote by $X_+ := X_{1}$, $X_- := X_{2}$ and by $X_0 := X_{(1,2)}$.

%%%%%
\subsection{Applications}

\begin{prop}\label{X_0 piece}
$\SL_3(\Z)$ is bi-exact towards $X_0$. 
\end{prop}
\begin{proof}
Since $P_0 := P_{(1,2)}$ is an amenable group (it is the upper triangular subgroup), the action $\Gamma \actson X_{0} = G/P_{0}$ is topologically amenable, and the result follows from Theorem \ref{bi-exact equivalence}.
\end{proof}

\begin{cor}\label{SLd solid}
Denote by $\Lambda$ either the top-left copy of $\SL_2(\Z)$ or the copy of $\Z^2$ inside $\Gamma := \SL_3(\Z)$:
\[\Lambda = \begin{pmatrix} *&*&0\\ *&*&0\\0&0&1\end{pmatrix} \simeq \SL_2(\Z) \quad \text{or} \quad \Lambda =  \begin{pmatrix} 1&0&*\\0&1&*\\0&0&1\end{pmatrix} \simeq \Z^2.\]
Then for any diffuse subalgebra $A$ of $L\Lambda$ the relative commutant $A' \cap L\Gamma$ is amenable.
\end{cor}
\begin{proof}
This relies on Theorem \ref{relatively solid}. In both cases we just have to check that whenever $(g_n)_n$ is a sequence of elements in $\Lambda$ which goes to infinity, the two ratios $s_1(g_n)/s_2(g_n)$ and $s_2(g_n)/s_3(g_n)$ go to infinity. In both cases, we actually prove the stronger fact that $1$ is a singular value of every element of $\Lambda$.
In the first case, this is obvious. In the second case, take $u \in \Z^2$. We compute
\[\begin{pmatrix} I_2 & 0\\u^t&1\end{pmatrix}\begin{pmatrix} I_2 & u\\0&1\end{pmatrix} - I_3 = \begin{pmatrix} I_2 & u\\u^t&1+u^tu \end{pmatrix} - I_3 = \begin{pmatrix} 0_2 & u\\u^t& u^tu \end{pmatrix}.\]
This matrix has rank at most $2$, which proves the claim.
\end{proof}

\begin{rem}
In contrast with the above corollary, we cannot apply directly Proposition \ref{X_0 piece} to the subgroup $\Lambda \simeq \GL_2(\Z) \ltimes \Z^2$ embedded as follows
\[(A,u) \in \GL_2(\Z) \ltimes \Z^2 \mapsto \begin{pmatrix} A&u\\0&\det(A) \end{pmatrix}.\]
Indeed for every integer $n \geq 1$, define the element $g_n \in \Lambda_n$ as follows
\[g_n = \begin{pmatrix} n&n-1&n\\n+1&n&0\\0&0&1\end{pmatrix}.\]
Then $(g_n)_n$ is a sequence which goes to infinity inside $\Lambda$, but none of its cluster points in $\partial\Gamma$ belongs to $X_0$. This last fact comes from the fact that $(g_n/n)_n$ converges to a rank $2$ matrix (so the two top singular values are equivalent to $n$). Nevertheless, a similar solidity property is expected to hold true. 
\end{rem}

\begin{prop}
For all $d \geq 3$, consider inside $\Gamma_d := \PSL_d(\Z)$ the following subgroup $\Lambda_d$ isomorphic to $\Z^{d-1}$: 
\[\Lambda_d := \left\{\pm\begin{pmatrix} I_{d-1}&u\\0&1\end{pmatrix} \, , \, u \in \Z^{d - 1}\right\} \simeq \Z^{d - 1}\]
Then $L(\Lambda_d)$ is a maximal abelian subalgebra inside $L(\Gamma_d)$ and the inclusion $L(\Lambda_3) \subset L(\Gamma_3)$ is not isomorphic with the inclusion $L(\Lambda_d) \subset L(\Gamma_d)$ for any $d \geq 4$.
\end{prop}
\begin{proof}
The first statement easily reduces to the lemma below which proves a relative ICC condition.
To prove the second part of the statement denote by $j: \Z^{d-1} \to \Lambda_d < \Gamma_d$ the described embedding and denote by $e_1,\dots,e_{d-1}$ the canonical basis of $\Z^{d-1}$. Then observe that for $d \geq 4$ the subgroup $j(\Z e_{d-1}) < \SL_d(\Z)$ is centralized by the (non-amenable) subgroup
\[\left\{\pm \begin{pmatrix} A&0\\0&I_2 \end{pmatrix} \, , \, A \in \SL_{d-2}(\Z)\right\}.\]
Hence this situation is distinct for the case of $d = 3$ in which Corollary \ref{SLd solid} applies.
\end{proof}

\begin{lem}
Fix $d \geq 3$ and denote by $\Gamma := \GL_d(\Z)$ and by $\Lambda \simeq \Z^{d - 1}$ the subgroup
\[\Lambda =  \left\{\begin{pmatrix} I_{d-1}&u\\0&1\end{pmatrix} \, , \, u \in \Z^{d - 1} \right\} \simeq \Z^{d - 1}.\]
Then for all $g \in \Gamma$ such that $\{sgs^{-1} \, , \, s \in \Lambda\}$ is finite we have either $ g \in \Lambda$, or $-g \in \Lambda$.
\end{lem}
\begin{proof}
Take $g \in \Gamma$ with coefficients $g_{i,j}$, $1 \leq i,j\leq d$. 
Denote by $j : \Z^{d-1} \to \Lambda < \Gamma$ the embedding described in the statement. Denote by $e_1,\dots,e_{d-1}$ the canonical basis of $\Z^{d-1}$. Then for $1 \leq k \leq d-1$, and any $n \geq 0$, we have
\[ [j(ne_k)gj(ne_k)^{-1}]_{k,k} = g_{k,k} + ng_{d,k}.\]
Therefore, if $g_{d,k} \neq 0$ for some $1 \leq k \leq d-1$, then we may conjugate $g$ by elements $j(ne_k)$ with $n$ arbitrarily large and obtain this way elements whose $(k,k)$-th coefficient is arbitrarily large. Hence the set $\{sgs^{-1} \, , \, s \in \Lambda\}$ is infinite in this case.

Otherwise, we may write $g = \begin{pmatrix} A&u\\0&a\end{pmatrix}$ with $A \in \GL_{d-1}(\Z)$, $a = \det(A) = \pm 1$ and $u \in \Z^{d-1}$. For all $n \geq 1$ and $v \in \Z^{d-1}$, we then have
\[j(nv)gj(nv)^{-1} = \begin{pmatrix}A& u+ n(Av -av)\\0 & a\end{pmatrix}.\]
 If $A \neq \pm I_{d-1}$ then we may choose $v \in \Z^{d-1}$ in such a way that $Av \neq \det(A)v$. Then the conjugacy class $\{sgs^{-1} \, , \, s \in \Lambda\}$ is infinite in this case. The remaining case is then $A = \pm I_{d-1}$ and $a = \det(A)$, which is clearly equivalent to either $g \in \Lambda$, or $-g \in \Lambda$.
\end{proof}

\bibliographystyle{amsalpha}
\bibliography{ref}

\providecommand{\bysame}{\leavevmode\hbox to3em{\hrulefill}\thinspace}
\providecommand{\MR}{\relax\ifhmode\unskip\space\fi MR }
\providecommand{\MRhref}[2]{%
  \href{http://www.ams.org/mathscinet-getitem?mr=#1}{#2}
}
\providecommand{\href}[2]{#2}
\begin{thebibliography}{GHW05}

\bibitem[AD87]{AD87}
Claire Anantharaman-Delaroche, \emph{Syst\`emes dynamiques non commutatifs et
  moyennabilit\'e}, Math. Ann. \textbf{279} (1987), no.~2, 297--315.

\bibitem[Arv77]{A77}
William Arveson, \emph{Notes on extensions of {$C^{\sp*}$}-algebras}, Duke
  Math. J. \textbf{44} (1977), no.~2, 329--355.

\bibitem[BC15]{BC14}
R\'emi Boutonnet and Alessandro Carderi, \emph{Maximal amenable von {N}eumann
  subalgebras arising from maximal amenable subgroups}, Geom. Funct. Anal.
  \textbf{25} (2015), no.~6, 1688--1705.

\bibitem[Bek90]{Be90}
Mohammed E.~B. Bekka, \emph{Amenable unitary representations of locally compact
  groups}, Invent. Math. \textbf{100} (1990), no.~2, 383--401.

\bibitem[BH18]{BH16}
R\'emi Boutonnet and Cyril Houdayer, \emph{Amenable absorption in amalgamated
  free product von neumann algebras}, Kyoto J. Math. \textbf{Advance
  publication} (2018), 11 pp.

\bibitem[BO08]{BO08}
Nathanial~P. Brown and Narutaka Ozawa, \emph{{$C^*$}-algebras and
  finite-dimensional approximations}, Graduate Studies in Mathematics, vol.~88,
  American Mathematical Society, Providence, RI, 2008.

\bibitem[Bou13]{Bo12}
R\'emi Boutonnet, \emph{{${\rm W}^*$}-superrigidity of mixing {G}aussian
  actions of rigid groups}, Adv. Math. \textbf{244} (2013), 69--90.

\bibitem[Bow99]{Bo96}
B.~H. Bowditch, \emph{Convergence groups and configuration spaces}, Geometric
  group theory down under ({C}anberra, 1996), de Gruyter, Berlin, 1999,
  pp.~23--54.

\bibitem[CH89]{CH89}
Michael Cowling and Uffe Haagerup, \emph{Completely bounded multipliers of the
  {F}ourier algebra of a simple {L}ie group of real rank one}, Invent. Math.
  \textbf{96} (1989), no.~3, 507--549.

\bibitem[CH10]{CH10}
Ionut Chifan and Cyril Houdayer, \emph{Bass-{S}erre rigidity results in von
  {N}eumann algebras}, Duke Math. J. \textbf{153} (2010), no.~1, 23--54.

\bibitem[Con80]{Co80}
A.~Connes, \emph{A factor of type {${\rm II}_{1}$} with countable fundamental
  group}, J. Operator Theory \textbf{4} (1980), no.~1, 151--153.

\bibitem[CP13]{chifanpeterson}
Ionut Chifan and Jesse Peterson, \emph{Some unique group-measure space
  decomposition results}, Duke Math. J. \textbf{162} (2013), no.~11,
  1923--1966.

\bibitem[CS13]{CS13}
Ionut Chifan and Thomas Sinclair, \emph{On the structural theory of {${\rm
  II}_1$} factors of negatively curved groups}, Ann. Sci. \'Ec. Norm. Sup\'er.
  (4) \textbf{46} (2013), no.~1, 1--33 (2013).

\bibitem[Dav96]{D96}
Kenneth~R. Davidson, \emph{{$C^*$}-algebras by example}, Fields Institute
  Monographs, vol.~6, American Mathematical Society, Providence, RI, 1996.
  \MR{1402012}

\bibitem[FM77]{FM77b}
Jacob Feldman and Calvin~C. Moore, \emph{Ergodic equivalence relations,
  cohomology, and von {N}eumann algebras. {II}}, Trans. Amer. Math. Soc.
  \textbf{234} (1977), no.~2, 325--359.

\bibitem[Fur76]{Fu76}
Harry Furstenberg, \emph{A note on {B}orel's density theorem}, Proc. Amer.
  Math. Soc. \textbf{55} (1976), no.~1, 209--212.

\bibitem[GHW05]{GHW05}
Erik Guentner, Nigel Higson, and Shmuel Weinberger, \emph{The {N}ovikov
  conjecture for linear groups}, Publ. Math. Inst. Hautes \'Etudes Sci. (2005),
  no.~101, 243--268.

\bibitem[Haa16]{Ha16}
Uffe Haagerup, \emph{Group {$C^*$}-algebras without the completely bounded
  approximation property}, J. Lie Theory \textbf{26} (2016), no.~3, 861--887.

\bibitem[Ioa11a]{Io11}
Adrian Ioana, \emph{Cocycle superrigidity for profinite actions of property
  ({T}) groups}, Duke Math. J. \textbf{157} (2011), no.~2, 337--367.

\bibitem[Ioa11b]{Io10}
\bysame, \emph{{$W^*$}-superrigidity for {B}ernoulli actions of property ({T})
  groups}, J. Amer. Math. Soc. \textbf{24} (2011), no.~4, 1175--1226.

\bibitem[Ioa15]{Io15}
\bysame, \emph{Cartan subalgebras of amalgamated free product {${\rm II}_1$}
  factors}, Ann. Sci. \'Ec. Norm. Sup\'er. (4) \textbf{48} (2015), no.~1,
  71--130, With an appendix by Ioana and Stefaan Vaes.

\bibitem[{Ioa}17]{Io17}
A.~{Ioana}, \emph{{Rigidity for von Neumann algebras}}, ArXiv e-prints (2017).

\bibitem[IPP08]{IPP08}
Adrian Ioana, Jesse Peterson, and Sorin Popa, \emph{Amalgamated free products
  of weakly rigid factors and calculation of their symmetry groups}, Acta Math.
  \textbf{200} (2008), no.~1, 85--153.

\bibitem[Mar91]{Ma91}
G.~A. Margulis, \emph{Discrete subgroups of semisimple {L}ie groups},
  Ergebnisse der Mathematik und ihrer Grenzgebiete (3), vol.~17,
  Springer-Verlag, Berlin, 1991.

\bibitem[MVN36]{MvN36}
F.~J. Murray and J.~Von~Neumann, \emph{On rings of operators}, Ann. of Math.
  (2) \textbf{37} (1936), no.~1, 116--229.

\bibitem[MvN43]{MvN43}
F.~J. Murray and J.~von Neumann, \emph{On rings of operators. {IV}}, Ann. of
  Math. (2) \textbf{44} (1943), 716--808.

\bibitem[OP10a]{OP10a}
Narutaka Ozawa and Sorin Popa, \emph{On a class of {${\rm II}_1$} factors with
  at most one {C}artan subalgebra}, Ann. of Math. (2) \textbf{172} (2010),
  no.~1, 713--749.

\bibitem[OP10b]{OP10b}
\bysame, \emph{On a class of {${\rm II}_1$} factors with at most one {C}artan
  subalgebra, {II}}, Amer. J. Math. \textbf{132} (2010), no.~3, 841--866.

\bibitem[Oza04]{Oz04}
Narutaka Ozawa, \emph{Solid von {N}eumann algebras}, Acta Math. \textbf{192}
  (2004), no.~1, 111--117.

\bibitem[Oza06a]{Oz06A}
\bysame, \emph{Amenable actions and applications}, International {C}ongress of
  {M}athematicians. {V}ol. {II}, Eur. Math. Soc., Z\"urich, 2006,
  pp.~1563--1580.

\bibitem[Oza06b]{Oz06}
\bysame, \emph{A {K}urosh-type theorem for type {$\rm II_1$} factors}, Int.
  Math. Res. Not. (2006), Art. ID 97560, 21.

\bibitem[Oza09]{Oz09}
\bysame, \emph{An example of a solid von {N}eumann algebra}, Hokkaido Math. J.
  \textbf{38} (2009), no.~3, 557--561.

\bibitem[Pet09]{Pe09}
Jesse Peterson, \emph{{$L^2$}-rigidity in von {N}eumann algebras}, Invent.
  Math. \textbf{175} (2009), no.~2, 417--433.

\bibitem[Pop06a]{Po01}
Sorin Popa, \emph{On a class of type {${\rm II}_1$} factors with {B}etti
  numbers invariants}, Ann. of Math. (2) \textbf{163} (2006), no.~3, 809--899.

\bibitem[Pop06b]{Po03}
\bysame, \emph{Strong rigidity of {$\rm II_1$} factors arising from malleable
  actions of {$w$}-rigid groups. {I}}, Invent. Math. \textbf{165} (2006),
  no.~2, 369--408.

\bibitem[Pop06c]{Po04}
\bysame, \emph{Strong rigidity of {$\rm II_1$} factors arising from malleable
  actions of {$w$}-rigid groups. {II}}, Invent. Math. \textbf{165} (2006),
  no.~2, 409--451.

\bibitem[Pop07]{Po07}
\bysame, \emph{Deformation and rigidity for group actions and von {N}eumann
  algebras}, 445--477.

\bibitem[PS12]{PS12}
Jesse Peterson and Thomas Sinclair, \emph{On cocycle superrigidity for
  {G}aussian actions}, Ergodic Theory Dynam. Systems \textbf{32} (2012), no.~1,
  249--272.

\bibitem[PV10]{PV09}
Sorin Popa and Stefaan Vaes, \emph{Group measure space decomposition of {${\rm
  II}_1$} factors and {$W^\ast$}-superrigidity}, Invent. Math. \textbf{182}
  (2010), no.~2, 371--417.

\bibitem[PV14a]{PV12a}
\bysame, \emph{Unique {C}artan decomposition for {$\rm II_1$} factors arising
  from arbitrary actions of free groups}, Acta Math. \textbf{212} (2014),
  no.~1, 141--198.

\bibitem[PV14b]{PV12b}
\bysame, \emph{Unique {C}artan decomposition for {$\rm II_{1}$} factors arising
  from arbitrary actions of hyperbolic groups}, J. Reine Angew. Math.
  \textbf{694} (2014), 215--239.

\bibitem[Ska88]{Sk88}
Georges Skandalis, \emph{Une notion de nucl\'earit\'e en {$K$}-th\'eorie
  (d'apr\`es {J}.\ {C}untz)}, $K$-Theory \textbf{1} (1988), no.~6, 549--573.

\bibitem[TD15]{TD15}
Robin Tucker-Drob, \emph{Invariant means and the structure of inner amenable
  groups}, preprint (2015).

\bibitem[Vae10]{Va10}
Stefaan Vaes, \emph{Rigidity for von {N}eumann algebras and their invariants},
  Proceedings of the {I}nternational {C}ongress of {M}athematicians. {V}olume
  {III}, Hindustan Book Agency, New Delhi, 2010, pp.~1624--1650.

\bibitem[Vae13]{Va13}
\bysame, \emph{One-cohomology and the uniqueness of the group measure space
  decomposition of a {${\rm II}_1$} factor}, Math. Ann. \textbf{355} (2013),
  no.~2, 661--696.

\bibitem[Zim84]{Zi84}
Robert~J. Zimmer, \emph{Ergodic theory and semisimple groups}, Monographs in
  Mathematics, vol.~81, Birkh\"auser Verlag, Basel, 1984.

\end{thebibliography}

\end{document}